\documentclass[a4paper,10pt]{article}
\usepackage[utf8]{inputenc}
\usepackage[english]{babel}

\usepackage{comment,xcomment,amssymb,amsmath,color}
\usepackage{rotating,xypic,array}
\usepackage{latexsym,mathtools,amsthm}
\usepackage{subfig}
\usepackage{tikz,fancyvrb}
\usepackage{booktabs}  
\usetikzlibrary{arrows.meta}
\usepackage{enumitem}
\usepackage{url} 
\usepackage{longtable}
\usepackage[section]{placeins}
\usepackage{multirow}

\makeatletter
\def\namedlabel#1#2{\begingroup
   \def\@currentlabel{#2}%
   \label{#1}\endgroup
}
\makeatother

\title{Indefinite Einstein metrics on nice Lie groups}
\author{Diego Conti and Federico A. Rossi}

\newtheorem{theorem}{Theorem}[section]
\newtheorem{lemma}[theorem]{Lemma}

\newtheorem{corollary}[theorem]{Corollary}
\newtheorem{proposition}[theorem]{Proposition}
\theoremstyle{definition}

\newtheorem{example}[theorem]{Example}
\theoremstyle{remark}
\newtheorem{remark}[theorem]{Remark}

\newcommand{\abs}[1]{\left\vert#1\right\vert}
\newcommand{\R}{\mathbb{R}}
\newcommand{\lie}[1]{\mathfrak{#1}}     
\newcommand{\g}{\lie{g}}
\newcommand{\Z}{\mathbb{Z}}
\newcommand{\Q}{\mathbb{Q}}

\newcommand{\hook}{\lrcorner\,}

\newcommand{\Gtwo}{\mathrm{G}_2}

\newcommand{\id}{\mathrm{Id}}   

\newcommand{\gl}{\lie{gl}}

\newcommand{\Span}[1]{\operatorname{Span}\left\{#1\right\}}
\newcommand{\tran}[1]{\hspace{.2mm}\prescript{t\hspace{-.5mm}}{}{#1}}
\DeclareMathOperator{\ric}{ric}

\DeclareMathOperator{\Aut}{Aut}

\DeclareMathOperator{\diag}{diag}

\DeclareMathOperator{\logsign}{logsign}

\DeclareMathOperator{\Tr}{tr}

\newcolumntype{C}{>{$}c<{$}}
\newcolumntype{L}{>{$}l<{$}}
\newcolumntype{R}{>{$}r<{$}}

\begin{document}
\VerbatimFootnotes
\maketitle

\begin{abstract}
We introduce a systematic method to produce left-invariant, non-Ricci-flat Einstein metrics of indefinite signature on nice nilpotent Lie groups. On a nice nilpotent Lie group, we give a simple algebraic characterization of non-Ricci-flat left-invariant Einstein metrics in both the class of metrics for which the nice basis is orthogonal and a more general class associated to order two permutations of the nice basis.

We obtain classifications in dimension 8 and, under the assumption that the root matrix is surjective, dimension 9; moreover, we prove that Einstein nilpotent Lie groups of nonzero scalar curvature exist in every dimension $\geq 8$.
\end{abstract}

\renewcommand{\thefootnote}{\fnsymbol{footnote}}
\footnotetext{This work was partially supported by GNSAGA of INdAM and by PRIN $2015$ ``Real and Complex Manifolds: Geometry, Topology and Harmonic Analysis''.}
\renewcommand{\thefootnote}{\arabic{footnote}}

In pseudoriemannian geometry, Einstein metrics are characterized by the condition that the Ricci operator is a multiple of the identity, and they are often regarded as ``optimal'' or ``least curved'' metrics on a fixed manifold, partly due to their variational nature as critical points of the total scalar curvature functional.

The construction of Einstein Riemannian metrics is a classical problem. The homogeneous case has been studied intensively, although a classification is yet to be achieved. Arguably, the simplest nontrivial examples are the isotropy irreducible spaces
classified by \cite{Wolf:TheGeometry}. The case of zero scalar curvature is trivial, as Ricci-flat homogeneous metrics are necessarily flat (see \cite{AlekseevskiKimelFel}). Homogeneous Einstein manifolds with positive scalar curvature are compact; both  necessary and sufficient conditions on a compact homogeneous space for the existence of an Einstein metric are known (see \cite{BohmWangZiller, Bohm:nonexistence}). In the case of negative curvature, all known examples are solvmanifolds, i.e. solvable Lie groups endowed with a left-invariant metric; such a metric can be identified with an inner product on the corresponding Lie algebra, which will also be called a (Riemannian) metric. An outstanding conjecture of Aleksveesky states that all negatively curved Einstein homogeneous metrics are of this type.

The structure of Einstein Riemannian solvmanifolds is well understood: by the work of Heber (\cite{Heber:noncompact}) and Lauret (\cite{Lauret}), their Lie algebra is the orthogonal, semidirect product of the nilradical with an abelian factor, and the restriction to the nilradical gives a nilsoliton metric, meaning that the Ricci operator has the form $\ric=c\id+D$ for some derivation $D$.  In modern language,  $D$ must coincide with the Nikolayevsky or pre-Einstein derivation, namely the semisimple derivation $D$ (unique up to automorphisms) satisfying $\Tr(D\circ\phi)=\Tr \phi$ for all derivations $\phi$; the nilsoliton condition implies that the Nikolayevsky derivation has positive eigenvalues (see  \cite{Nikolayevsky}). Conversely, any  nilsoliton has an Einstein solvable extension defined by the Nikolayevsky derivation.

Thus, the classification of Riemannian Einstein solvmanifolds is reduced to the classification of nilsolitons. Nilsolitons are classified up to dimension $7$ in \cite{Will:RankOne,FernandezCulma:classification2}. A classification up to dimension $8$ was obtained in \cite{KadiogluPayne:Computational} under the assumption that the Nikolayevsky derivation has eigenvalues of multiplicity one. This condition implies that the Lie algebra has a nice basis as defined in \cite{LauretWill:Einstein}; nice bases can be characterized by the condition that all diagonal metrics have diagonal Ricci operator (see \cite{LauretWill:diagonalization}).

\smallskip
This paper is mostly concerned with the indefinite case; from now on, the metrics under consideration will be pseudoriemannian, and the corresponding scalar products will be of arbitrary signature. There are many differences with the positive-definite case; to begin with, indefinite homogeneous Ricci-flat metrics are not necessarily flat. Remaining in the nilpotent context, early examples of Ricci-flat metrics go back to \cite{Petrov:EinsteinSpaces} (see also \cite{Derdzinski}), and as observed in \cite{globke}  any bi-invariant metric on a nilpotent Lie group will be Ricci-flat, though generally not flat. Ricci-flat metrics with holonomy $\Gtwo^*$ were obtained in \cite{FinoLujan:TorsionFreeG22}.

Another important difference is the existence of nilpotent Lie algebras with indefinite Einstein metrics of nonzero scalar curvature. Whilst these are nilsoliton metrics with $D=0$, the standard metric induced on the (trivial) solvable extension defined by $D$  is not Einstein (see also Remark~\ref{remark:solvable_extension}); this suggests that the interplay between indefinite Einstein solvmanifolds and the geometry of nilmanifolds is more complicated compared to the Riemannian case.

Einstein nilmanifolds of nonzero scalar curvature are trickier to construct  than their Ricci-flat counterparts. Indeed, a necessary condition for existence is that the Nikolayevsky derivation is zero (equivalently, all Lie algebra derivations are trace-free; see~\cite{ContiRossi:EinsteinNilpotent}); there are only $11$ Lie algebras of dimension $7$ that satisfy this condition, and none in lower dimensions; even when the condition is satisfied, the computations involved in determining whether an Einstein metric actually exists are extremely complicated. The first examples, of dimension $8$, were  obtained in \cite{ContiRossi:EinsteinNilpotent} by considering a particular Lie algebra with a high degree of symmetry.

\smallskip
In this paper we illustrate a new, systematic method to construct Einstein nilmanifolds with nonzero scalar curvature, obtained as compact quotients $\Gamma\backslash G$ of nilpotent Lie groups with a left-invariant metric. The nilpotent Lie algebras we consider are those that admit a nice basis in the sense of \cite{LauretWill:Einstein}; this condition, although apparently quite special, is satisfied by most nilpotent Lie algebras, at least in low dimensions (see \cite[Section 3]{ContiRossi:Construction}).

Nice nilpotent Lie algebras of dimension $\leq 9$ are classified in \cite{ContiRossi:Construction} up to a natural notion of equivalence. The classification uses an algorithm based on a graph $\Delta$ and a matrix $M_\Delta$ associated to each nice Lie algebra, called its nice diagram and root matrix. The root matrix is a well-known object that reflects metric properties of diagonal metrics; for instance,  nilsoliton metrics can be expressed in terms of a linear system involving $M_\Delta$ and the structure constants (see \cite[Theorem 1]{Payne:ExistenceOfSolitonMetrics}). Nice diagrams were primarily introduced for classification purposes, but they also have a use in describing the symmetries in the structure of the nice Lie algebra.

We obtain a useful characterization of  diagonal Einstein metrics with nonzero scalar curvature on a fixed nice nilpotent Lie algebra in terms of the root matrix; indeed, such metrics can be identified with particular solutions to the linear system $\tran{X}M_\Delta=(1,\dots, 1)$; the existence of solutions is equivalent to a known necessary condition, namely the vanishing of the Nikolayevsky derivation (Lemma~\ref{lemma:TracelessDerivations}). A solution $X$ determines an actual Einstein metric if and only if it satisfies a polynomial condition relating $X$, the structure constants and the group of diagonal automorphisms of the  Lie algebra (see Theorem~\ref{thm:sigmacompatible}).

In the special case that the root matrix is surjective when viewed as a matrix with coefficients in $\Z_2$, the second condition is equivalent to the unique solution $X$  not being contained in any coordinate hyperplane: thus, the existence of a diagonal  Einstein metric can be determined purely considering a linear system (Theorem~\ref{thm:Sufficient_Condition}). Moreover, when the solution $X$ is contained in a coordinate hyperplane a diagonal Einstein metric can be found on a contraction limit (Proposition~\ref{prop:surjective_type_contraction_limit}). As a consequence, we find that any nice Lie algebra with invertible root matrix over $\Z_2$ has an Einstein contraction limit with nonzero scalar curvature (Corollary~\ref{cor:invertible}). We use this fact to show that Einstein nilmanifolds of nonzero scalar curvature exist in any dimension $\geq8$ (Theorem~\ref{thm:existinalldimensions}). Note that surjectivity of $M_\Delta$ over $\R$ is equivalent to  the Gram matrix $M_\Delta \tran{M_\Delta}$ being invertible; this condition has been considered in \cite{KadiogluPayne:Computational}.

Using the classification of \cite{ContiRossi:Construction}, we classify diagonal Einstein metrics of nonzero scalar curvature on nice nilpotent Lie algebras of dimension $8$ (Theorem~\ref{thm:EinsteinMetrics8}), and those with surjective root matrix in dimension 9 (Theorem~\ref{thm:EinsteinMetrics9}).

More generally, we consider a class of indefinite metrics associated to an involution $\sigma$ of the nice diagram $\Delta$; these metrics generalize diagonal metrics in the sense that the Ricci operator is diagonal. Our characterization of diagonal Einstein metrics of nonzero scalar curvature extends to this more general class (see Theorem~\ref{thm:sigmacompatible}); we obtain classifications analogous to those given in the diagonal case (see Theorems~\ref{thm:EinsteinMetrics8} and~\ref{thm:EinsteinMetrics9}).

All the nilpotent Lie algebras appearing in our classification have rational structure constants with respect to an appropriately chosen nice basis; by \cite{Malcev}, each corresponding Lie group $G$ admits a compact quotient $\Gamma\backslash G$, on which an  Einstein metric is induced.

\section{Nice Lie algebras and nice diagrams}
\label{sec:diagrams}
In this section we recall the definition of nice Lie algebra and the basic related facts; we refer to \cite{ContiRossi:Construction} for further details.

On a Lie algebra $\g$, a basis  $\{e_1,\dotsc, e_n\}$ with dual basis $\{e^1,\dotsc, e^n\}$ is called \emph{nice} if each $[e_i,e_j]$ is a multiple of a single basis element $e_k$ depending on $i,j$, and each $e_i\hook de^j$ is a multiple of a single $e^h$, depending on $i,j$. A \emph{nice nilpotent Lie algebra} consists of a real nilpotent Lie algebra with a fixed nice basis. Two nice nilpotent Lie algebras are regarded as \emph{equivalent} if there is an isomorphism that maps basis elements to multiples of basis elements. Up to equivalence, nice nilpotent Lie algebras of dimension $n$ can be viewed as particular elements of $\Lambda^2(\R^n)^*\otimes\R^n$, corresponding to nilpotent Lie algebra structures on $\R^n$ for which the standard basis is nice; the group $\Sigma_n$ of permutations in $n$ letters and the group of diagonal real matrices $D_n$ act naturally as equivalences.

To each nice nilpotent Lie algebra we can associate a directed graph $\Delta$ as follows:
\begin{itemize}
\item the nodes of $\Delta$ are the elements of the nice basis;
\item there is an arrow from $e_i$ to $e_j$ if $e_j$ is a nonzero multiple of some  $[e_i,e_h]$, i.e. $e_i\hook de^j\neq0$;
\item each arrow is endowed with a label that belongs to the set of nodes; if $e_j$ is a nonzero multiple of some  $[e_i,e_h]$, then the arrow from $e_i$ to $e_j$ has label $e_h$. By matter of notation, we will write $i\xrightarrow{h}j$.
\end{itemize}
The resulting graph $\Delta$ is called a \emph{nice diagram}; by construction, it is acyclic, and any two arrows can have at most one of source, destination and label in common (for the full list of conditions, we refer to \cite{ContiRossi:Construction}).  An \emph{isomorphism} of nice diagrams is an isomorphism of graphs, i.e. a pair of compatible bijections between nodes and arrows that also preserve labels. We will denote by $\Aut(\Delta)$ the group of automorphisms of a nice diagram; note that by construction $\Aut(\Delta)$ is a subgroup of $\Sigma_n$. It is clear that equivalent nice Lie algebras determine isomorphic nice diagrams.

The correspondence between nice nilpotent Lie algebras and nice diagrams is not one to one. Indeed, inequivalent nice nilpotent Lie algebras with the same nice diagram can appear in continuous families (though inequivalent nice bases on a fixed nilpotent Lie algebra can only appear in discrete families, see \cite[Corollary 3.7]{ContiRossi:Construction}). Nice Lie algebras associated to a nice diagram $\Delta$ are parametrized as follows.

Let $\mathcal{I}_\Delta$ be the set of the $I=\{\{i,j\},k\}$ such that $i\xrightarrow{j}k$; we shall write
\[E_I=e^{ij}\otimes e_k, \quad I=\{\{i,j\},k\},\ i<j,\]
where $e^{ij}$ stands for $e^i\wedge e^j$. Take the $D_n$-representation $V_\Delta$ freely generated by the $E_I$, $I\in\mathcal{I}_\Delta$. An actual Lie algebra structure on $\R^n$ is determined by an element
\[c=\sum_{I\in \mathcal{I}_\Delta} c_IE_I,\]
encoding the structure constants. Whenever $I=\{\{i,j\},k\}$, we shall write
\[c_{ijk}=\begin{cases}
c_I,& i<j\\ 0 & i> j
          \end{cases};
\]
explicitly, the Lie algebra structure satisfies
\[[e_i,e_j]=c_{ijk} e_k.\]
Here and in the sequel, we will use the Einstein convention for summation over repeated indices $i,j,k$; however, in this case the nice condition implies that the sum has a single nonzero term.

Equivalence classes of nice nilpotent Lie algebras with diagram $\Delta$ are para\-metrized by elements of
\[V_\Delta / (D_n\ltimes\Aut(\Delta)).\]

Each nice diagram $\Delta$ has an associated \emph{root matrix} defined as follows. Define a total order on $\mathcal{I}_\Delta$ by
 \[\{\{i,j\},k\}< \{\{l,m\},h\} \iff (k < h) \vee( k=h\wedge i<l),\quad i<j, l<m.\]
Notice that the two elements of $\mathcal{I}_\Delta$ coincide when $k=h$ and $i=l$ by the nice condition.  Identifying the vectors $e^1,\dotsc,e^n$ of $(\R^n)^*$ with row vectors of size $n$, we can associate to each $\{\{i,j\},k\}$ the row vector $-e^i-e^j+e^k$. The root matrix $M_\Delta$ is the $m\times n$ matrix with rows determined by the elements of $\mathcal{I}_\Delta$ taken in increasing order.

Our choice of sign in the definition of $M_\Delta$ is motivated by the fact that rows of the root matrix represent the weights for the action of $D_n$ on $E_I$, i.e.
\[\alpha_I(\lambda_1e^1\otimes e_1+\dots + \lambda_ne^n\otimes e_n)=-\lambda_i-\lambda_j+\lambda_k, \quad I=\{\{i,j\},k\}.\]

 We can associate to $M_\Delta$ a Lie algebra homomorphism $M_\Delta^D\colon d_n\to d_m$. This homomorphism realizes the correspondence between the natural action of $D_n$ on $V_\Delta\subset\Lambda^2T^*\otimes T$ and the action of $D_m$ via the diagram
\begin{equation}
\label{eqn:diag_exp_M_delta}
\xymatrix{
d_n\ar[d]_\exp \ar[r]^{M_\Delta^D} & d_m\ar[d]^\exp \\
D_n \ar[r]^{e^{M_\Delta}} & D_m
}
\end{equation}
An element of $d_n$ is a Lie algebra derivation when it maps each $E_I$, $I\in\mathcal{I}_\Delta$ to zero. Thus, $\ker M_\Delta^D$ is the space of diagonal derivations; the group of equivalences is
\[H=\ker e^{M_\Delta}.\]

The  isomorphism between the group of invertible integers $\Z^*$ and $\Z_2$ can be extended to a function
\[\logsign \colon\R^*\to \Z_2, \quad \logsign(x)=\begin{cases}
                                                  0 & x>0 \\
                                                  1 & x<0
                                                 \end{cases},
\]
where the notation is justified by the identity
\[x=e^{\log\abs{x}} (-1)^{\logsign{x}}.\]
With this language, an element $\epsilon=(\epsilon_1,\dotsc, \epsilon_n)$ of $((\Z^*)^n)$ lies in $H$ if and only if 
\[\logsign\epsilon=(\logsign\epsilon_1,\dotsc, \logsign\epsilon_n)\]
lies in the kernel of the matrix $M_{\Delta,2}$ obtained by projecting the (integer) entries of $M_\Delta$ to $\Z_2$.
\begin{remark}
There is a different well-known method to attach a nice nilpotent Lie algebra to an undirected graph, where each node and edge define an element of the nice basis, and nonzero Lie brackets correspond to pairs of nodes connected by an edge (see e.g. \cite{LauretWill:Einstein}). The resulting Lie algebra is always two-step, and therefore admits no Einstein metric of nonzero scalar curvature (see~\cite{ContiRossi:EinsteinNilpotent}).
\end{remark}

\section{Construction of Einstein metrics}
\label{sec:einstein}
We are interested in left-invariant pseudoriemannian metrics on a nice nilpotent Lie group $G$; such metrics are one-in-to-one correspondence with scalar products (i.e. symmetric, nondegenerate bilinear forms) on the corresponding Lie algebra~$\g$.

The Ricci operator of the pseudoriemannian metric, as a linear map $\ric\colon\g\to\g$, can  be expressed in terms of the scalar product and  structure constants as follows. Consider the natural contractions
\[\langle\,, \rangle\colon (\Lambda^2\g^*\otimes \g)\otimes  (\Lambda^2\g \otimes \g^*)\to\R,\qquad \langle\,, \rangle\colon (\g^*\otimes \g)\otimes (\g^*\otimes \g)\to\R,\]
defined independently of the metric. Let $c=c_IE_I\in V_\Delta$; by \cite{ContiRossi:EinsteinNilpotent}, the contraction of the Ricci operator with  $A\in\g^*\otimes \g$ satisfies
\begin{equation}
 \label{eqn:jensen_ricci}
\langle \ric,A\rangle=\frac12 \langle A c, q(c)\rangle,
\end{equation}
where
\[q\colon\Lambda^2\g^*\otimes\g\to \Lambda^2\g\otimes \g^*, \qquad q(\alpha\wedge\beta\otimes v)=\alpha^\sharp\wedge \beta^\sharp\otimes v^\flat.\]
Notice that the coefficient $\frac12$, as opposed to the coefficient $\frac14$ present in  \cite[Theorem 2.8]{ContiRossi:EinsteinNilpotent}, stems from the equality
\[1=\langle e^i\wedge e^j, e_i\wedge e_j\rangle =\frac12\langle e^i\otimes e^j - e^j\otimes e^i,e_i\otimes e_j-e_j\otimes e_i\rangle.\]
Explicitly, the components of the Ricci operator are determined by the structures constants $c_{ijk}$, the metric tensor $g_{ij}$ and its inverse $g^{ij}$ by
\[
\ric_h^k=\frac14   g^{im}g^{ln}g_{hp}c_{ilk}c_{mnp} -\frac12 g^{km}g^{jn}g_{ip}c_{hji}c_{mnp}.
\]

The Einstein condition means that the Ricci operator is a multiple of the identity; we are interested in the case where the Ricci operator is nonzero. Up to a normalization, Einstein metrics with positive scalar curvature can be identified with solutions of
\begin{equation}
\tag{E}
 \label{eqn:normalized_einstein}
 \ric=\frac12\id.
\end{equation}
Einstein metrics of negative scalar curvature are obtained by changing the sign of the metric, with the effect of reversing the signature. We proved in  \cite{ContiRossi:EinsteinNilpotent} that metrics satisfying \eqref{eqn:normalized_einstein} are obstructed by derivations of nonzero trace; in this section we show that this condition depends only on the underlying diagram. Restricting to the case of diagonal metrics, we also give two other necessary conditions for the existence of a diagonal metric satisfying \eqref{eqn:normalized_einstein}. We also give sufficient conditions that apply to situations in which the number of nonzero brackets is small relative to the dimension.

We will denote by $[1]$ any column vector with all entries equal to $1$; the number of components will be implied by the context.
\begin{lemma}
\label{lemma:TracelessDerivations}
Let $\g$ be a nice Lie algebra with diagram $\Delta$. Then all derivations of $\g$ are traceless if and only if $[1]$ is in the space spanned by the columns of $\tran M_\Delta$.
\end{lemma}
\begin{proof}
It is a consequence of the nice condition that the diagonal part of a derivation is again a derivation (see \cite[Section 4]{Nikolayevsky} and \cite[Corollary~3.15]{LauretDere:OnRicciNegative}); thus, it suffices to prove the statement for diagonal derivations.

A matrix $\lambda_i e^i\otimes e_i$ in $d_n$ is a derivation when it is annihilated by each weight $\alpha_I$ (see e.g. \cite{Payne:ExistenceOfSolitonMetrics}). Let $\alpha_{\Tr}$ be the weight of the one-dimensional representation $\Lambda^n\R^n$, i.e.
\[\alpha_{\Tr}(\lambda_i e^i\otimes e_i)=\lambda_1+\dotsc + \lambda_n.\]
The condition that diagonal derivations are traceless is equivalent to
\[\ker
\begin{pmatrix}
\alpha_{I_1}\\
\vdots\\
\alpha_{I_m}
\end{pmatrix}=\ker
\begin{pmatrix}
\alpha_{I_1}\\
\vdots\\
\alpha_{I_m}\\
\alpha_{\Tr}
\end{pmatrix};\]
clearly, this is true when $\alpha_{\Tr}$ is a linear combination of the $\alpha_I$. Written in the basis $\{e^1\otimes e_1,\dots,e^n\otimes e_n\}$, this condition means that $\tran{[1]}$ is in the span of the rows of the root matrix, i.e. $[1]$ is in the span of its columns.
\end{proof}
\begin{remark}
\label{remark:solvable_extension}
Recall from \cite{Nikolayevsky} that every Lie algebra $\g$ admits a semisimple derivation  $\phi$ with real eigenvalues such that
\[\Tr(\phi\circ\psi)=\Tr \psi, \quad \psi\in \operatorname{Der}(\g);\]
this derivation is unique up to automorphisms and it is known as the \emph{Nikolayevsky} (or \emph{pre-Einstein}) \emph{derivation}.

If derivations are traceless then $\phi=0$ satisfies the defining condition for the Nikolayevsky derivation, and conversely, $\phi=0$ implies that all derivations are traceless. Therefore, the obstruction of Lemma~\ref{lemma:TracelessDerivations} is equivalent to the vanishing of the Nikolayevsky derivation.

It follows that, given a nilpotent Lie algebra $\g$ with a metric satisfying \eqref{eqn:normalized_einstein}, its solvable extension defined by the Nikolayevsky derivation admits no metric satisfying \eqref{eqn:normalized_einstein}. Indeed, this extension is a product $\g\times\R$, which is unimodular with Killing form zero. Since the center of $\g\times\R$ is not contained in its derived algebra, the existence of Einstein metrics with nonzero scalar curvature is ruled out by \cite[Lemma 2.2]{ContiRossi:EinsteinNilpotent}.
\end{remark}

Recall from \cite{LauretWill:Einstein} that diagonal metrics on a nice nilpotent Lie algebra have a diagonal Ricci tensor, i.e. the elements of the nice basis are eigenvectors. This fact extends to a more general class of metrics. Indeed, let $\Delta$ be a nice diagram and let $\sigma\in\Sigma_n$ be a permutation of the set of nodes such that $\sigma^2=\id$. We shall say that a metric on $\g$ is \emph{$\sigma$-diagonal} if the metric tensor has the form
\[g=g_i e^i\otimes e^{\sigma_i},\]
so that $e_i^\flat=g_ie^{\sigma_i}$ and $(e^i)^\sharp=\frac1{g_i} e_{\sigma_i}$.
If $\sigma$ is the identity, this condition means that the metric is diagonal. In this paper, we will only consider the case where  $\sigma$ is in $\Aut(\Delta)$, i.e. $\sigma_i\xrightarrow{\sigma_j}\sigma_k$ whenever $i\xrightarrow{j}k$. An element of $\Aut(\Delta)$ that squares to the identity will be called a \emph{diagram involution}.

For a fixed diagram involution $\sigma$, there are two relevant groups of symmetries acting on the space of $\sigma$-diagonal metrics. The first is the natural action of $D_n$ on $\g^*\otimes \g^*$ induced by the action on $\g$, which reflects the natural ambiguity in the choice of a nice basis; the subgroup of $D_n$ that preserves the Lie algebra structure is $\ker e^{M_\Delta}$. To introduce the second, we denote by $\Z^*=\{\pm1\}$ the group of invertible integers; then $\sigma$ acts on $(\Z^*)^n$ and its fixed point set is a group $((\Z^*)^n)^\sigma$. We consider the action of $((\Z^*)^n)^\sigma$ on the space of $\sigma$-diagonal metrics given by
\[ (\epsilon_1,\dotsc, \epsilon_n)(g_1e^1\otimes e^{\sigma_1}+\dots + g_ne^n\otimes e^{\sigma_n})=\epsilon_1g_1e^1\otimes e^{\sigma_1}+\dots +\epsilon_ng_ne^n\otimes e^{\sigma_n};\]
notice that this is not the (trivial) action obtained by the natural action of $D_n$ by restriction. By restriction, we obtain an action of
\[G_\sigma= ((\Z^*)^n)^\sigma\cap  H.\]
We will see that the combined action of $G_\sigma\times H$ preserves \eqref{eqn:normalized_einstein} (see Theorem~\ref{thm:sigmacompatible}).

An automorphism $\sigma$ of $\Delta$ induces naturally an isomorphism $\tilde\sigma\colon V_\Delta\to V_\Delta$, i.e. $\tilde\sigma(e^{ij}\otimes e_k)=e^{\sigma_i,\sigma_j}\otimes e_{\sigma_k}$. Given an element of $V_\Delta$ with components $c_I$, we will denote by $\tilde c_I$ the components of its image, i.e.
\[\sum \tilde c_I E_I =\tilde \sigma(\sum c_IE_I).\]

\begin{theorem}
\label{thm:sigmacompatible}
Let $\Delta$ be a nice diagram, let $c\in V_\Delta$ define a Lie algebra structure and let $\sigma$ be a diagram involution. The Lie algebra admits a $\sigma$-diagonal metric satisfying \eqref{eqn:normalized_einstein} if and only if
\begin{enumerate}[label=(\roman*)]
\item\label{thm:sigmacompatible:cond:unoeqlineare}
$\tran M_\Delta X =[1]$
for some $X=(x_I)\in \R^{\mathcal{I}_\Delta}$ ;
\item\label{thm:sigmacompatible:cond:dueimmagine} $\left(\dfrac{x_I}{c_I\tilde{c}_I}\right)=e^{M_\Delta}(g)$ for some $g\in D_n$;
\item\label{thm:sigmacompatible:cond:tresigma} $\sigma g=g$.
\end{enumerate}
If conditions~\ref{thm:sigmacompatible:cond:unoeqlineare}--\ref{thm:sigmacompatible:cond:tresigma} hold, any metric in the $G_\sigma\times H$-orbit of
\[g_i e^i\otimes e^{\sigma_i},\quad g=\diag(g_1,\dotsc, g_n)\]
satisfies \eqref{eqn:normalized_einstein}; the $G_\sigma\times H$-orbit is determined uniquely by $X$.
\end{theorem}
\begin{proof}
Write the general $\sigma$-diagonal metric as $g_i e^i\otimes e^{\sigma_i}$, where $g=\diag(g_1,\allowbreak \dots,g_n)$ is $\sigma$-invariant.

If $I=\{\{i,j\},k\}$, write
\[E^\sigma_I=e_{\sigma_i,\sigma_j}\otimes e^{\sigma_k}, \quad g_I=\frac{g_k}{g_ig_j},\]
so that $q(E_I)=g_IE^\sigma_I$.
By equation \eqref{eqn:jensen_ricci}, for any $A\in \gl(n,\R)$
\[\langle \ric, A\rangle =\frac12\sum_{I,J\in\mathcal{I}_\Delta}\langle c_I (A E_I),c_Jg_J E^\sigma_J\rangle.\]
Since $\sigma$ is an automorphism, and due to the properties of a nice diagram, two weight vectors $e^{ij}\otimes e_k$ and $e^{\sigma_i,\sigma_j}\otimes e_h$ can only appear in $V_\Delta$ if $k=\sigma_h$; similarly, if both $e^{ij}\otimes e_k$ and $e^{\sigma_i,h}\otimes e_{\sigma_k}$ appear then  $h=\sigma_j$. This implies that $\langle AE_I,E^\sigma_J\rangle$ is zero when the diagonal part of $A$ is zero, or when  $I$ and $J$ are not related by $\sigma$.

In other words, the Ricci operator is diagonal, and we have
\begin{multline*}
2\ric^h_{h} =2\langle \ric, e^h\otimes e_h\rangle = \langle e^h\otimes e_h \cdot c_{ijk}e^i\wedge e^j\otimes e_k, \sigma(c)_{i,j,k}  \frac {g_k}{g_ig_j} e_{i}\wedge e_{j}\otimes e^{k}\rangle \\
= c_{ijk}\tilde\sigma(c)_{ijk} \frac {g_k}{g_ig_j} (\delta_{hk}-\delta_{hi}-\delta_{hj})
=x_{ijk}(\delta_{hk}-\delta_{hi}-\delta_{hj}),
 \end{multline*}
where $X=(x_I)$ is defined as in condition~\ref{thm:sigmacompatible:cond:dueimmagine} (see also \cite[Theorem 8]{Payne:ExistenceOfSolitonMetrics} in the Riemannian case). It follows that $\tran{X}M_\Delta =
2(\ric^1_{1},\dots,\ric^n_{n})$, and therefore \eqref{eqn:normalized_einstein} is equivalent to conditions~\ref{thm:sigmacompatible:cond:unoeqlineare}--\ref{thm:sigmacompatible:cond:tresigma}.

Moreover, it is clear that condition~\ref{thm:sigmacompatible:cond:tresigma} is invariant under $G_\sigma\times H$. Now let $h$ be another $\sigma$-diagonal metric compatible with $X$; then, $e^{M_\Delta}(g_1,\dotsc, g_n)=e^{M_\Delta}(h_1,\dotsc, h_n)$, i.e.
\[
 f=\begin{pmatrix}g_1/h_1 \\ \vdots \\ g_n/h_n\end{pmatrix} \in \ker e^{M_\Delta}.
\]
In particular, if we define $M_{\Delta,2}$ as the matrix with entries in $\Z_2$ obtained by projecting the integer matrix $M_\Delta$, we have
\[M_{\Delta,2}(\logsign f)=\logsign e^{M_\Delta}(f)=0;\]
in other words, $f\in G_\sigma$.  Therefore, we can assume up to the action of $G_\sigma$ that each $g_i$ has the same sign as $h_i$, so that $f=\exp(t)$, $t\in \ker M_\Delta$. Then $\exp(t/2)$ defines a diagonal automorphism that maps $g$ into $h$.
\end{proof}
\begin{remark}
Theorem~\ref{thm:sigmacompatible} illustrates a phenomenon which is specific to metrics of indefinite signature: indeed,  Einstein Riemannian metrics on a nilpotent Lie algebras only exist if the Lie algebra is abelian (see \cite{Milnor:curvatures}), and the theorem becomes a trivial statement for positive-definite signature. Nevertheless, it might be instructive to specialize the statement to the positive-definite case for comparison with the literature.

To this end, we assume that $g$ has positive entries and $\sigma=\id$, so that the metric is diagonal and Riemannian, and condition \ref{thm:sigmacompatible:cond:tresigma} is trivially satisfied. In this case,
Theorem~\ref{thm:sigmacompatible} states that an Einstein, diagonal Riemannian metric exists if and only if  $\tran M_\Delta X=[1]$ for some $X=(x_{ijk})$ satisfying the equations
\[\frac{x_{ijk}}{c_{ijk}^2}=\frac{g_k}{g_ig_j}.\]
In particular, all entries in $X$ should be positive. Unsurprisingly, this cannot occur, as  $\tran M_\Delta X=[1]$ implies that
\[ -\sum x_{ijk}=\tran [-1]X=(\tran[1]\tran M_\Delta) X = \tran[1][1]>0.\]

The presence of condition~\ref{thm:sigmacompatible:cond:unoeqlineare} implies that we are considering nilpotent Lie algebras with Nikolayevsky derivation equal to zero; in the  Riemannian setting, the Nikolayevsky derivation measures the difference between the Ricci tensor of a nilsoliton metric and the Ricci tensor of an Einstein metric (see \cite{Nikolayevsky}), so in the case at hand the two conditions become formally equivalent.

By a theorem of Nikolayevsky (\cite[Theorem 3]{Nikolayevsky}), a nice nilpotent Lie algebra admits a Riemannian nilsoliton metric if and only if  the linear system $M_\Delta \tran M_\Delta X=[-1]$ has a solution with positive entries. By~\cite[Theorem~3.1]{Payne:Applications}, the Nikolayevsky derivation corresponds to the vector $\tran M_\Delta b+[1]$ where $b$ is any solution of $M_\Delta \tran M_\Delta b=[1]$; therefore, under the assumption that the Nikolayevsky derivation is zero,  $\tran M_\Delta X=[1]$ becomes equivalent to  $M_\Delta \tran M_\Delta X=[-1]$. Thus, Theorem~\ref{thm:sigmacompatible} is formally analogous to Nikolayevsky's theorem, with the linear inequalities $x_{ijk}>0$ replaced by the more complicated condition \ref{thm:sigmacompatible:cond:dueimmagine}.

In the pseudoriemannian case, condition \ref{thm:sigmacompatible:cond:dueimmagine}  still implies certain linear inequalities in the $x_{ijk}$ depending on the signature (see next Corollary~\ref{cor:necessary_condition}), but these linear inequalities do not imply \ref{thm:sigmacompatible:cond:dueimmagine} (see Example~\ref{example:nonsurjective} below). In fact, Nikolayevsky's proof uses a common feature of nilsolitons and Einstein metrics, namely their characterization as critical points of the scalar curvature functional on appropriately defined spaces of metrics (see \cite{Jensen:scalar,Nikolayevsky}). In the positive definite case this functional has the form $\sum_I b_I e^{\langle \cdot, w_I\rangle}$, where $w_I$ is the $I$-th row of $M_\Delta$ and $b_I=\frac12c_I^2$. In the pseudoriemannian case a similar formula holds, but the $b_I$ have a sign that depends on the signature; convexity, which is the key property in the proof of  \cite[Theorem 3]{Nikolayevsky}, is therefore lost.
\end{remark}

\begin{remark}
\label{rem:fintaFamigliaParam}
In analogy with the simple case, where the canonical Einstein metric is typically an isolated point (see \cite[Theorem 12.3]{DerdzinskiGal}), it is natural to ask whether $\sigma$-diagonal solutions of \eqref{eqn:normalized_einstein} can appear in continuous families on a fixed Lie algebra. Since the action of $H=\ker e^{M_\Delta}$ on the space of metrics amounts to a nice Lie algebra equivalence, two metrics $g,g'$ should be regarded as genuinely different if $e^{M_\Delta}(g)$ differs from $e^{M_\Delta}(g')$. In particular, one can ask if the map $e^{M_\Delta}$ is transverse to the affine subspace in $\R^{\mathcal{I}_\Delta}$ defined by the equations
\[\tran M_\Delta {(c_I\tilde c_I y_I)}=[1].\]
This would imply the existence of a solution  $X$ to $\tran M_\Delta {X}=[1]$ and some $v\in d_n$ with $M_\Delta^D(v)X$ a nonzero element of $\ker\tran{M_\Delta}$.

We have not found this condition to be verified for any of the Einstein metrics appearing in this paper.
\end{remark}
\begin{remark}
The statement of Theorem~\ref{thm:sigmacompatible} can be modified to characterize Ricci-flat $\sigma$-diagonal  metrics by replacing $[1]$ with the zero vector (see~\cite{ContiRossi:RicciFlat}).
\end{remark}

If we consider diagonal metrics, corresponding to the case $\sigma=\id$, we obtain necessary conditions that only depend on the diagram:
\begin{corollary}
\label{cor:necessary_condition}
Let $\g$ be a nice Lie algebra with diagram $\Delta$.  If $\g$ admits a diagonal metric satisfying \eqref{eqn:normalized_einstein}, then there exists a vector $X=(x_I)\in \R^{\mathcal{I}_\Delta}$ such that
\begin{description}
\item[($\mathbf{K}$)\namedlabel{eqn:linear_system}{($\mathbf{K}$)}]
$\tran M_\Delta X=[1]$;
\item[($\mathbf{H}$)\namedlabel{necessarycondition2}{($\mathbf{H}$)}] each component $x_I$ is not zero;
\item[($\mathbf{L}$)\namedlabel{necessarycondition3}{($\mathbf{L}$)}] the vector $(\logsign x_I)\in  (\Z_2)^{\mathcal{I}_\Delta}$ is in the image of $M_{\Delta,2}$.
\end{description}
\end{corollary}
\begin{proof}
\ref{eqn:linear_system} and~\ref{necessarycondition2} follow immediately from Theorem~\ref{thm:sigmacompatible}. Setting $\sigma=\id$ in Theorem~\ref{thm:sigmacompatible}, condition~\ref{thm:sigmacompatible:cond:tresigma} implies
\begin{equation}
\label{eqn:polynomial_eqns}
\biggl(\frac{x_I}{c^2_I}\biggr)_{I\in\mathcal{I}_\Delta}=e^{M_\Delta}(g),
\end{equation}
and in particular
\[\logsign X= \logsign e^{M_\Delta}(g) = M_{\Delta,2}\logsign g.\qedhere\]
\end{proof}
\begin{remark}
It is possible to generalize the construction to arbitrary Lie algebras with a fixed basis, defining labeled diagrams and weight vectors $c_IE_I\in V_\Delta$ in a similar way as in Section~\ref{sec:diagrams}. Whilst~\ref{eqn:linear_system},~\ref{necessarycondition2} and~\ref{necessarycondition3} remain necessary in this case, the conditions of Theorem~\ref{thm:sigmacompatible}  are not sufficient  because the Ricci operator is not necessarily diagonal.
\end{remark}

\begin{remark}
In the more general situation of a diagram involution $\sigma$ and a $\sigma$-diagonal metric, we have by construction
\begin{equation}
 \label{eqn:epsilonsigma}
\sigma M_\Delta\sigma = M_\Delta,
\end{equation}
where $\sigma$ also denotes the permutation matrix corresponding to the action of $\sigma\in\Aut(\Delta)$ on $V_\Delta$. Thus, if~\ref{eqn:linear_system} has a solution, it also admits a $\tau$-invariant solution. Constructing a $\sigma$-diagonal metric that satisfies \eqref{eqn:normalized_einstein} is then a matter of solving
\begin{equation*}
\biggl(\frac{x_I}{c_I\tilde{c}_I}\biggr)_{I\in\mathcal{I}_\Delta}=e^{M_\Delta}(g)
\end{equation*}
with $g$ $\sigma$-invariant. Conditions~\ref{eqn:linear_system} and~\ref{necessarycondition2} also apply to this case; however the analogue of~\ref{necessarycondition3} turns out to depend on the structure constants. We refer to \cite{ContiRossi:RicciFlat} for more details.
\end{remark}

On the constructive side, we can split the problem in two: first, one solves the linear system~\ref{eqn:linear_system}, verifying whether there exists a solution satisfying conditions~\ref{necessarycondition2} and~\ref{necessarycondition3}. Then, one has to solve the polynomial equations \eqref{eqn:polynomial_eqns}. If one is only interested in proving the existence of an Einstein metric, one can in some cases use the following result to skip the second step:
\begin{theorem}\label{thm:Sufficient_Condition}
Let $\g$ be a nice nilpotent Lie algebra with diagram $\Delta$; if $M_{\Delta,2}$ is surjective and $X$ is a vector satisfying conditions \ref{eqn:linear_system} and \ref{necessarycondition2} in Corollary~\ref{cor:necessary_condition}, then there exists a diagonal metric satisfying \eqref{eqn:normalized_einstein}.
\end{theorem}
\begin{proof}
Consider the diagram given in \eqref{eqn:diag_exp_M_delta}. Since $M_{\Delta,2}$ is surjective, $e^{M_\Delta}(D_n)$ intersects each connected component of $D_m$, as can be seen by considering the image of elements of the form $\diag(\pm1,\dots, \pm1)$.

In addition, the rows of $M_\Delta$ are linearly independent over $\Z_2$, hence over $\Q$; thus, $M_\Delta$ is also surjective over $\R$ and $e^{M_\Delta}$ is an epimorphism. It follows that  \eqref{eqn:polynomial_eqns} has a solution.
\end{proof}
Motivated by this result, we will say that $\g$ is of \emph{surjective type} when  $M_{\Delta,2}$ is surjective. This condition has remarkable consequences in the classification of nice Lie algebras: if $\g$ is a nice nilpotent Lie algebra of surjective type, any nice Lie algebra with the same diagram is equivalent to $\g$ (see \cite{ContiRossi:Construction}).

\begin{example}
\label{example:nonsurjective} 
The obvious fact that surjective matrices cannot have more rows than columns limits the scope of Theorem~\ref{thm:Sufficient_Condition} to nice Lie algebras where the  number of nonzero brackets does not exceed the dimension. For an example where the number of brackets is less than the dimension, consider
\[\texttt{952:355} = (0,0,0,0,e^{12},e^{34},e^{13}+e^{24},e^{15}+e^{23},e^{14}+e^{36}),\]
where the string \texttt{952} corresponds the lower central series dimensions and the full label \texttt{952:355} refers to the classification of \cite{ContiRossi:Construction}.

The corresponding $8\times 9$ matrix  is
\begin{equation}
\label{eqn:nonsurjective}
M_\Delta = \left(
\begin{array}{ccccccccc}
 -1 & -1 & 0 & 0 & 1 & 0 & 0 & 0 & 0 \\
 0 & 0 & -1 & -1 & 0 & 1 & 0 & 0 & 0 \\
 -1 & 0 & -1 & 0 & 0 & 0 & 1 & 0 & 0 \\
 0 & -1 & 0 & -1 & 0 & 0 & 1 & 0 & 0 \\
 -1 & 0 & 0 & 0 & -1 & 0 & 0 & 1 & 0 \\
 0 & -1 & -1 & 0 & 0 & 0 & 0 & 1 & 0 \\
 -1 & 0 & 0 & -1 & 0 & 0 & 0 & 0 & 1 \\
 0 & 0 & -1 & 0 & 0 & -1 & 0 & 0 & 1 \\
\end{array}
\right);
\end{equation}
it has maximal rank over $\Z_2$ and $\tran{X}=(-6,-6,4,-3,-7,8,8,-7)$ satisfies the hypotheses of Theorem~\ref{thm:Sufficient_Condition}. Thus, this Lie algebra carries a solution of~\eqref{eqn:normalized_einstein}; the metric is computed in Example~\ref{ex:952:355DiagEinstein}.
\end{example}

When the number of brackets exceeds the dimension, the root matrix is not surjective and system \eqref{eqn:polynomial_eqns} contains more equations than unknowns, but may still admit a solution. Examples of diagonal Einstein metrics with nonsurjective $M_\Delta$ appear in Table~\ref{table:8DimNiceEinsteinDiagonal} (\texttt{86532:6}, \texttt{842:121a} and \texttt{842:121b}).

\subsection{Lie algebras without Einstein metrics}

\begin{table}[thp]
{\setlength{\tabcolsep}{2pt}
\centering
\caption{9-dimensional nice nilpotent Lie algebras with surjective $M_\Delta$  not satisfying \ref{necessarycondition2} or \ref{necessarycondition3}\label{9dimObstructedSurjective} and their upper central series dimensions}
\begin{footnotesize}
\begin{tabular}{>{\ttfamily}c C C c}
\toprule
\textnormal{Name} & \g & \text{UCS Dim.} & \text{Obstruction}\\
\midrule
\multicolumn{4}{c}{Type: $32211$}\\
96421:82a&0,0,0,e^{12},e^{13},e^{14},e^{15}+e^{23},e^{16},e^{18}+e^{24}+e^{35}&24679 &\ref{necessarycondition3}\\
96421:82b&0,0,0,- e^{12},e^{13},e^{14},e^{15}+e^{23},e^{16},e^{18}+e^{24}+e^{35}&24679 &\ref{necessarycondition3}\\
96421:84a&0,0,0,e^{12},e^{13},e^{24},e^{15}+e^{23},e^{26},e^{14}+e^{28}+e^{35}&24679 &\ref{necessarycondition3}\\
96421:84b&0,0,0,- e^{12},e^{13},e^{24},e^{15}+e^{23},e^{26},e^{14}+e^{28}+e^{35}&24679 &\ref{necessarycondition3}\\
\multicolumn{4}{c}{Type: $3222$}\\
9642:113a&0,0,0,e^{12},e^{13},e^{14},e^{15},e^{16}+e^{23},e^{17}+e^{24}+e^{35}&2469 &\ref{necessarycondition3}\\
9642:113b&0,0,0,- e^{12},e^{13},e^{14},e^{15},e^{16}+e^{23},e^{17}+e^{24}+e^{35}&2469 &\ref{necessarycondition3}\\
9642:115a&0,0,0,e^{12},e^{13},e^{14},e^{35},e^{16}+e^{23},e^{15}+e^{24}+e^{37}&2469 &\ref{necessarycondition3}\\
9642:115b&0,0,0,e^{12},- e^{13},e^{14},e^{35},e^{16}+e^{23},e^{15}+e^{24}+e^{37}&2469 &\ref{necessarycondition3}\\
\multicolumn{4}{c}{Type: $3321$}\\
9631:87a&0,0,0,e^{12},e^{13},e^{23},e^{14},e^{15}+e^{36},e^{17}+e^{24}+e^{35}&2579 &\ref{necessarycondition3}\\
9631:87b&0,0,0,- e^{12},e^{13},e^{23},e^{14},e^{15}+e^{36},e^{17}+e^{24}+e^{35}&2579 &\ref{necessarycondition3}\\
9631:93a&0,0,0,e^{12},e^{13},e^{23},e^{14},e^{15}+e^{26},e^{17}+e^{24}+e^{35}&2579 &\ref{necessarycondition3}\\
9631:93b&0,0,0,e^{12},- e^{13},e^{23},e^{14},e^{15}+e^{26},e^{17}+e^{24}+e^{35}&2579 &\ref{necessarycondition3}\\
9631:93c&0,0,0,- e^{12},e^{13},e^{23},e^{14},e^{15}+e^{26},e^{17}+e^{24}+e^{35}&2579 &\ref{necessarycondition3}\\
9631:93d&0,0,0,- e^{12},- e^{13},e^{23},e^{14},e^{15}+e^{26},e^{17}+e^{24}+e^{35}&2579 &\ref{necessarycondition3}\\
\multicolumn{4}{c}{Type: $423$}\\
953:166a&0,0,0,0,e^{12},e^{13},e^{15}+e^{23},e^{16}+e^{24},e^{14}+e^{25}+e^{36}&369 &\ref{necessarycondition3}\\
953:166b&0,0,0,0,- e^{12},e^{13},e^{15}+e^{23},e^{16}+e^{24},e^{14}+e^{25}+e^{36}&369 &\ref{necessarycondition3}\\
953:169a&0,0,0,0,e^{12},e^{13},e^{15}+e^{23},e^{16}+e^{34},e^{14}+e^{25}+e^{36}&369 &\ref{necessarycondition3}\\
953:169b&0,0,0,0,e^{12},e^{13},e^{15}+e^{23},-e^{16}+e^{34},e^{14}+e^{25}+e^{36}&369 &\ref{necessarycondition3}\\
953:169c&0,0,0,0,- e^{12},e^{13},e^{15}+e^{23},e^{16}+e^{34},e^{14}+e^{25}+e^{36}&369 &\ref{necessarycondition3}\\
953:169d&0,0,0,0,- e^{12},e^{13},e^{15}+e^{23},-e^{16}+e^{34},e^{14}+e^{25}+e^{36}&369 &\ref{necessarycondition3}\\
\multicolumn{4}{c}{Type: $432$}\\
952:384&0,0,0,0,e^{12},e^{34},e^{14}+e^{23},e^{13}+e^{25},e^{15}+e^{24}+e^{36}&359 &\ref{necessarycondition2}\\
952:732a&0,0,0,0,e^{12},e^{13},e^{34},e^{15}+e^{24}+e^{36},e^{14}+e^{25}+e^{37}&259 &\ref{necessarycondition3}\\
952:732b&0,0,0,0,e^{12},e^{13},e^{34},-e^{15}+e^{24}+e^{36},e^{14}+e^{25}+e^{37}&259 &\ref{necessarycondition3}\\
952:733a&0,0,0,0,e^{12},e^{13},e^{34},e^{14}+e^{25}+e^{36},e^{15}+e^{24}+e^{37}&259 &\ref{necessarycondition3}\\
952:733b&0,0,0,0,e^{12},e^{13},e^{34},e^{36}+e^{25}-e^{14},e^{15}+e^{24}+e^{37}&259 &\ref{necessarycondition3}\\
952:733c&0,0,0,0,- e^{12},e^{13},e^{34},e^{14}+e^{25}+e^{36},e^{15}+e^{24}+e^{37}&259 &\ref{necessarycondition3}\\
952:733d&0,0,0,0,- e^{12},e^{13},e^{34},-e^{14}+e^{25}+e^{36},e^{15}+e^{24}+e^{37}&259 &\ref{necessarycondition3}\\
\end{tabular}
\end{footnotesize}
}
\end{table}

In this section we collect some examples of nice Lie algebras that do not admit diagonal Einstein metrics. These examples show that the necessary conditions~\ref{eqn:linear_system},~\ref{necessarycondition2} and~\ref{necessarycondition3} are independent, in the sense that \ref{eqn:linear_system} does not imply \ref{necessarycondition2} and \ref{necessarycondition2} does not imply \ref{necessarycondition3}, and not sufficient for the existence of an Einstein metric. In addition, we give an example of a nice diagram with a nontrivial automorphism $\sigma$ such that conditions~\ref{thm:sigmacompatible:cond:unoeqlineare} and \ref{thm:sigmacompatible:cond:dueimmagine} of Theorem~\ref{thm:sigmacompatible} can be satisfied but not condition~\ref{thm:sigmacompatible:cond:tresigma}.

\begin{remark}
We have already pointed out that there exist nice nilpotent Lie algebras that do not satisfy condition~\ref{eqn:linear_system}: indeed, it is not satisfied by any nice nilpotent Lie algebras of dimension $\leq 7$ (see~\cite{ContiRossi:Construction}; in fact, by Lemma~\ref{lemma:TracelessDerivations} condition~\ref{eqn:linear_system} is the same obstruction considered in \cite{ContiRossi:EinsteinNilpotent}).
\end{remark}

\begin{example}[Hyperplane obstruction \ref{necessarycondition2}] 
Take the Lie algebra \texttt{952:384}:
\[(0,0,0,0,e^{12},e^{34},e^{14}+e^{23},e^{13}+e^{25},e^{15}+e^{24}+e^{36}).\]
It is easy to see that the associated matrix $M_{\Delta}$
has maximal rank over $\Z_2$. It turns out that 
\[(-6, -6, -3, 4, 8, -7, 0, 8, -7) M_{\Delta}=(1,\dotsc, 1),\] 
but this solution lies in the coordinate hyperplane $x_{159}=0$. Thus, \ref{necessarycondition2} is not satisfied; in this case, any diagonal solution of \eqref{eqn:normalized_einstein} would have to satisfy $g_9/(g_1 g_5) = 0$, which is impossible. Note that this is the only Lie algebra in dimension $9$ with $M_{\Delta}$ surjective that is obstructed by condition \ref{necessarycondition2}.

The same obstruction can also occur for nonsurjective $M_{\Delta}$. For example, take the Lie algebra \texttt{9531:495a}: 
\[(0,0,0,0,- e^{12},-e^{13}+e^{24},e^{15}+e^{34},e^{14}+e^{26}+e^{35},e^{17}+e^{23}+e^{46}).\]
The general solution of $\tran XM_{\Delta}=(1,\dots, 1)$ is
\begin{multline*}
\tran X=(x_{125}, x_{136}, -3, 0, -8 - x_{125} - x_{136}, 8,
\\ -6 - x_{125}, -1 +x_{125}, -9 - x_{125} - x_{136}, 8, 2 + x_{125} + x_{136}),
\end{multline*}
which is contained in the coordinate hyperplane $x_{157}=0$.
\end{example}

\begin{example}[Sign obstruction~\ref{necessarycondition3}]
Take the Lie algebra \texttt{9631:93a}:
\[(0,0,0,e^{12},e^{13},e^{23},e^{14},e^{15}+e^{26},e^{17}+e^{24}+e^{35}).\]

It is easy to see that the associated matrix $M_\Delta$
has maximal rank over $\R$ but not over $\Z_2$. It turns out that the unique solution of  $\tran XM_{\Delta}=(1,\dots, 1)$ is 
\[\tran X=(5, 14, -13, -17, 15, -14, -18, 21, -2).\]
However, $\logsign X$ is not in the image of $M_{\Delta,2}$; in other words, the equations
\[\frac{g_4}{g_1 g_2}=5,\ \frac{g_5}{g_1 g_3}=14,\ \frac{g_9}{g_2g_4}=21,\ \frac{g_9}{g_3g_5}=-2\]
do not admit a solution.

For an example where  $M_{\Delta}$ is not surjective, consider the Lie algebra \texttt{9521:217a}:
\[(0,0,0,0,e^{12},e^{14}+e^{23},e^{13}+e^{24},e^{15},e^{18}+e^{25}+e^{37}+e^{46});\]
the matrix $M_\Delta$ has maximal rank over $\R$ but not over $\Z_2$. In this case the general solution is $\tran{X}=(5, 25/2, x_{236}, 25/2,-25 - x_{236},  -15, -16,19, -27/2 - x_{236}, 23/2 + x_{236})$; no element in this affine space satisfies
~\ref{necessarycondition3}, reflecting the fact that the equations
\begin{gather*}
\frac{g_{5}}{g_{1} g_{2}}=5,\ \frac{2 g_{7}}{g_{1} g_{3}}=25,\ \frac{g_6}{g_3g_2}+\frac{g_7}{g_4g_2}=-25,\ \frac{g_{9}}{g_{2} g_{5}}=19,\ \frac{2 g_{6}}{g_{2} g_{3}}+\frac{2 g_{9}}{g_{3} g_{7}}=-27\\
\frac{2 g_{6}}{g_{1} g_{4}}=25,\ \frac{g_{8}}{g_{1} g_{5}}=-15,\ \frac{g_{9}}{g_{1} g_{8}}=-16,\ \frac{g_{9}}{g_{2} g_{5}}=19,\ \frac{2 g_{6}}{g_{2} g_{3}}+23=\frac{2 g_{9}}{g_{4} g_{6}}
\end{gather*}
do not admit a solution (solving the first five equations in $g_1, g_5, g_6, g_7, g_9$ and substituting in the last one leads to $\frac{38 g_2^2+5 g_3^2}{76 g_2^2+135 g_3^2}=0$).

The same argument applies to all nice Lie algebras with the same diagrams, since the obstructions only depend on the diagram. For other examples, see Table~\ref{9dimObstructedSurjective}.
\end{example}

\begin{example}[Nonsurjective $M_{\Delta}$ satisfying~\ref{eqn:linear_system},~\ref{necessarycondition2} and~\ref{necessarycondition3} but with no solution of~\eqref{eqn:polynomial_eqns}]
Take the Lie algebra \texttt{96421:147a}:
\[
(0,0,0,e^{12},e^{13},e^{24},e^{15}+e^{23},e^{14}+e^{26},e^{16}+e^{28}+e^{35}).
\]
It is easy to see that $M_\Delta$ (and consequently $M_{\Delta,2}$) is not surjective, but there exists a vector $X$ that satisfies ~\ref{eqn:linear_system},~\ref{necessarycondition2} and~\ref{necessarycondition3}.
However, the associated system~\eqref{eqn:polynomial_eqns} is the following:
\begin{gather*}
\frac{2 g_4}{g_1 g_2}=11,\ \frac{2 g_5}{g_1 g_3}+21=0,\ \frac{g_7}{g_1 g_5}=-10,\\
\frac{g_7}{g_2 g_3}=11,\ \frac{2 g_8}{g_2 g_6}+21=0,\ \frac{2 g_9}{g_3 g_5}+3=0,\ \frac{g_1 g_6 +g_2 g_8}{g_1 g_2 g_4}=\frac{9}{2},\\
\frac{2 g_6}{g_2 g_4}+19=\frac{2 g_9}{g_1 g_6},\ \frac{g_6}{g_4g_2}+\frac{g_9}{g_8g_2}+7=0.
\end{gather*}
and it does not admit any real solution (one can use the first seven equations to eliminate  $g_2,g_4,g_5,g_6,g_7,g_8,g_9$; the resulting polynomial equations do not admit any real solutions in $g_1$, $g_3$). This example shows that conditions~\ref{eqn:linear_system},~\ref{necessarycondition2} and~\ref{necessarycondition3} are not sufficient for the existence of a solution to \eqref{eqn:normalized_einstein}.
\end{example}

Finally, in the next example we show that  condition~\ref{thm:sigmacompatible:cond:tresigma} of Theorem~\ref{thm:sigmacompatible} is independent of the others.
\begin{example}[Obstruction to the existence of a $\sigma$-diagonal Einstein metric]\label{example:NoSigmaMetric}
Consider the Lie algebra \texttt{952:355} of Example~\ref{example:nonsurjective}. 
Its diagram has an automorphism $\sigma=(1,3)(2,4)(5,6)(8,9)$ of order $2$. Following Theorem~\ref{thm:sigmacompatible}, let $\tran{X}=(-6,-6,4,-3,-7,8,8,-7)$ be the unique solution of $\tran{X} M_\Delta =(1,\dotsc,1)$; then \eqref{eqn:polynomial_eqns} reads:
\begin{gather*}
\frac{g_{5}}{g_{1} g_{2}}=-6,\ \frac{g_{6}}{g_{3} g_{4}}=-6,\ \frac{g_{7}}{g_{1} g_{3}}=-4,\ \frac{g_{7}}{g_{2}g_4}=3,\\
\frac{g_{8}}{g_{1} g_{5}}=-7,\ \frac{g_{8}}{g_{2} g_{3}}=-8,\ \frac{g_{9}}{g_{1} g_{4}}=-8,\ \frac{g_{9}}{g_{3} g_{6}}=-7,
\end{gather*}
but a $\sigma$-diagonal metric $g$ should satisfy $g_1=g_3$, $g_2=g_4$, giving
\[0<g_{7}=3g_2^2,\quad0>g_{7}=-4g_3^2,\]
which is a contradiction. Therefore, this Lie algebra admits no $\sigma$-diagonal solution to \eqref{eqn:normalized_einstein} (although it admits a diagonal Einstein metric, see Table~\ref{table:9DimSurjectiveNormal}). This shows that the necessary condition~\ref{thm:sigmacompatible:cond:tresigma} of Theorem~\ref{thm:sigmacompatible} is independent of the first two conditions.
\end{example}

\section{Contraction limits and central extensions}
Given an element $\sum c_IE_I\in V_\Delta$ that defines a Lie algebra, it is clear that the elements in $V_\Delta$ that define equivalent nice Lie algebras are the elements of the orbit of $\sum c_IE_I$ for the natural action of $\Sigma_n\ltimes D_n$ on $V_\Delta$. The group of diagonal matrices $D_n$ acts on $V_\Delta$ via $e^{M_\Delta}$. Thus, up to permutations, equivalent nice Lie algebras are related by an element of $e^{M_\Delta}(D_n)$; this operation amounts to a component-wise rescaling of the nice basis. Elements in the closure \[\overline{e^{M_\Delta}(D_n)\bigl(\sum c_IE_I\bigr)}\subseteq V_\Delta\] still define nice nilpotent Lie algebras, by continuity. The corresponding diagram could differ from $\Delta$ by having less arrows; if this is the case, the resulting Lie algebra will be said to be a \emph{contraction limit} of $\g$ (see \cite{Goze:OnTheVarieties}).

\begin{proposition}
\label{prop:surjective_type_contraction_limit}
Let $\g$ be a nice nilpotent Lie algebra of surjective type such that $\tran{M_\Delta}X=[1]$ admits a solution. Then either $\g$ or a unique contraction limit of $\g$ has a diagonal metric satisfying \eqref{eqn:normalized_einstein}.
\end{proposition}
\begin{proof}
If none of the entries of $X$ is zero, we can apply Theorem~\ref{thm:Sufficient_Condition}. Otherwise, we can consider the element
\[\sum_{x_I\neq0} E_I\in V_\Delta;\]
because $e^{M_\Delta}$ is surjective, this is a contraction limit of $\g$ with diagram $\Delta'$. The resulting matrix $M_{\Delta',2}$ is still surjective over $\Z_2$, because it is obtained from $M_{\Delta,2}$ by removing some rows; therefore, we can apply Theorem~\ref{thm:Sufficient_Condition}.

Uniqueness follows from the fact that any Einstein contraction limit leads to the same vector $X$, so the indices $I$ that survive in the limit must be exactly those for which $x_I\neq0$.
\end{proof}

\begin{corollary}
\label{cor:invertible}
Let $\g$ be a nice nilpotent Lie algebra with diagram $\Delta$; if $M_{\Delta,2}$ is invertible, then either $\g$ or a unique contraction limit of $\g$ has a diagonal metric satisfying \eqref{eqn:normalized_einstein}.
\end{corollary}
\begin{proof}
If $M_{\Delta,2}$ is invertible, then
\[1=\det M_{\Delta,2}=\det M_\Delta\mod 2;\]
therefore $M_\Delta$ is invertible and we can apply Proposition~\ref{prop:surjective_type_contraction_limit} with $X$ given by $X=\tran{M_\Delta^{-1}}[1]$.
\end{proof}

\begin{example}
\label{esempio:MaxRank}
Consider the 8-dimensional nice Lie algebra \texttt{842:117}:
\[(0,0,0,0,e^{12},e^{34},e^{15}+e^{24}+e^{36},e^{13}+e^{25}+e^{46}).\]
Its root matrix is
\[M_\Delta=\left(
\begin{array}{cccccccc}
 -1 & -1 & 0 & 0 & 1 & 0 & 0 & 0 \\
 0 & 0 & -1 & -1 & 0 & 1 & 0 & 0 \\
 -1 & 0 & 0 & 0 & -1 & 0 & 1 & 0 \\
 0 & -1 & 0 & -1 & 0 & 0 & 1 & 0 \\
 0 & 0 & -1 & 0 & 0 & -1 & 1 & 0 \\
 -1 & 0 & -1 & 0 & 0 & 0 & 0 & 1 \\
 0 & -1 & 0 & 0 & -1 & 0 & 0 & 1 \\
 0 & 0 & 0 & -1 & 0 & -1 & 0 & 1 \\
\end{array}
\right)\]
which has determinant $9$. In this case $\tran{X}=(-5,-5,-3,7,-3,7,-3,-3)$, so Corollary~\ref{cor:invertible} implies the existence of an Einstein metric of nonzero scalar curvature.
\end{example}

\begin{example}
Consider the Lie algebra $\g$ with structure constants
\[(0,0,0,0,e^{12},e^{34},e^{13}+e^{24},e^{15}+e^{23},e^{14}+e^{36}+e^{25}).\] 
In this case, the root matrix has determinant $9$, and so satisfies the hypothesis of Corollary~\ref{cor:invertible}; however, the resulting vector $X$ turns out to have a component equal to zero, corresponding to the bracket $[e_2,e_5]=e_9$. Therefore, $\g$ does not have any diagonal metric satisfying \eqref{eqn:normalized_einstein}, and its only contraction limit that carries such a metric is obtained by setting $[e_2,e_5]=0$, giving back the Lie algebra of Example~\ref{example:nonsurjective}.

Notice that up to reordering the rows the root matrix of $\g$ can be written in the form
\[
 M_{\Delta'}= \begin{pmatrix} M_\Delta \\ \alpha \end{pmatrix}, \qquad
 \alpha = \begin{pmatrix}
            0 & -1 & 0 & 0 & -1 & 0 & 0&0 & 1 \\
          \end{pmatrix},
\]
with $M_\Delta$ given by \eqref{eqn:nonsurjective}, so the fact  that $(1,\dots, 1)$ is in the span of the first $8$ rows actually follows from the computations of Example~\ref{example:nonsurjective}.

We observe that $\g$ is equivalent to \texttt{952:384} of Table~\ref{9dimObstructedSurjective}. 
\end{example}

Another natural operation on a nice nilpotent Lie algebra $\g$ is taking a central extension (see e.g. \cite{Gong}); in general, this means constructing an exact sequence
\[0\to \R^k\to \lie{h}\to \g\to 0\]
where $\R^k$ denotes the $k$-dimensional Abelian Lie algebra, and its image in $\lie{h}$ is contained in the center. Since we are interested in Einstein metrics, it is natural to require the center to be contained in the derived Lie algebra (see \cite[Lemma~3.1]{ContiRossi:EinsteinNilpotent}); in terms of diagrams, we add $k$ nodes and require each to be reached by a pair of arrows starting from the existing nodes.

We will concentrate on the case $k=1$. More precisely, given $e_i,e_j$ in the nice basis such that
\begin{equation}
 \label{eqn:needed_for_central_extension}
[e_i,e_j]=0, \quad de^{ij}=0,
\end{equation}
we add a node $r$ to the diagram and declare $de^r=e^{ij}$. Notice that the Jacobi condition in the form $d^2=0$ is automatically satisfied. We will say that the resulting Lie algebra is obtained by a \emph{one-bracket extension}.

\begin{lemma}
Any one-bracket extension of a nice nilpotent Lie algebra of surjective type is also a nice nilpotent Lie algebra of surjective type.
\end{lemma}
\begin{proof}
Let $\g$ be of surjective type and let $\g'$ be a one-bracket extension; it is clear that $\g'$ is nice and nilpotent. The root matrices are related by
\[
M_{\Delta'}=\begin{pmatrix}
 M_\Delta & 0\\
 * & 1
\end{pmatrix}.\qedhere\]
\end{proof}
\begin{remark}
An $n$-dimensional nice Lie algebra $\g$ with diagram $\Delta$ can be realized as a one-bracket extension of a nice Lie algebra of dimension $n-1$ if and only if $M_\Delta$ has a column whose only nonzero entry equals $1$, or equivalently if a node in $\Delta$ has two incoming arrows and no outgoing arrow; indeed, if $e_r$ is in the center, then $\g$ is a central extension of the Lie algebra $\g/\Span{e_r}$.

It would be interesting to classify nice nilpotent Lie algebras of surjective type that cannot be realized as one-bracket extensions.
\end{remark}

If in addition to \eqref{eqn:needed_for_central_extension} the one-form  $e^i$ is closed, the extension procedure can be iterated replacing $e^j$ with $e^r$, yielding a tower of one-bracket extensions of surjective type. This leads to the following:
\begin{theorem}
\label{thm:existinalldimensions}
For each $n\geq8$, there exist $n$-dimensional nice nilpotent Lie algebras with a diagonal metric satisfying \eqref{eqn:normalized_einstein}.
\end{theorem}
\begin{proof}
Let $\g$ be the  $8$-dimensional Lie algebra of Example~\ref{esempio:MaxRank}; it is straightforward to verify that $M_{\Delta,2}$ is invertible.

Define a one-bracket extension $\g_1$ by setting $de^{9}=e^{14}$ (appearing as \texttt{952:394} in Table~\ref{table:9DimSurjectiveNormal}); inductively, extend $\g_k$ to $\g_{k+1}$ by declaring $de^{k+8}=e^1\wedge e^{k+7}$. By construction, the root matrix $M_{\Delta_k,2}$ is invertible. By Corollary~\ref{cor:invertible}, for each $k$ either $\g_k$ or a contraction limit of $\g_k$ has a diagonal metric satisfying \eqref{eqn:normalized_einstein}.
\end{proof}

\section{Explicit Einstein metrics}
In dimension $8$ there are only $4$ diagrams, corresponding to $6$ inequivalent nice Lie algebras, that satisfy the necessary condition \ref{eqn:linear_system} for the existence of a solution to \eqref{eqn:normalized_einstein}, i.e. where all derivations are traceless. It turns out that they all admit a diagonal Einstein metric. These Lie algebras and their Einstein metrics are listed in Table~\ref{table:8DimNiceEinsteinDiagonal}, where the structure constant are relative to an orthonormal nice basis, i.e. satisfying $\langle e_i,e_j\rangle=\pm\delta_{ij}$. Note that the Lie algebra \texttt{842:121b} and its Einstein diagonal metrics were already present in \cite[Theorem 5.2]{ContiRossi:EinsteinNilpotent} as the first examples of non-Ricci-flat Einstein invariant metrics on a nilpotent Lie group. 

Theorem~\ref{thm:sigmacompatible} can be used directly to construct these diagonal Einstein metrics, as shown in the following example.

\begin{example}[Less nonzero brackets than dimension, $M_\Delta$ surjective]\label{ex:952:355DiagEinstein}  
Take the Lie algebra \texttt{952:355} of Example~\ref{example:nonsurjective}.
The system \eqref{eqn:polynomial_eqns} has the one-parameter family of solutions
\begin{gather*}
g_1 = \frac{4}{21},\ g_3 = \frac{4}{21},\ g_4= -\frac{64}{1323 g_2},\ g_5= -\frac{8}{7} g_2,\\
g_6=\frac{512}{9261 g_2},\ g_7=\frac{64}{441},\ g_8=\frac{32 g_2}{21},\ g_9= -\frac{2048}{27783 g_2}.
\end{gather*}
Notice however that all these metrics are related by an equivalence, more precisely an element of $\ker e^{M_\Delta}$, to a solution with $g_2=\pm1$; thus, there are essentially two distinct solutions to \eqref{eqn:normalized_einstein} on this Lie algebra. Indeed, there is a unique vector $X$ satisfying Theorem~\ref{thm:Sufficient_Condition}, because $\tran{M_\Delta}$ is injective, and $X$ determines the metric up to equivalence and the action of $G_\id=\ker M_{\Delta,2}$.
\end{example}
In the same way, Theorem~\ref{thm:sigmacompatible} can be used to determine $\sigma$-diagonal Einstein metrics, as illustrated in the following example.

\begin{example}\label{ex:8dimSIGMA} 
Consider the Lie algebra \texttt{842:117} (see Example~\ref{esempio:MaxRank})
which admits the automorphism $\sigma =(1,3)(2,4)(5,6)$. Fix the nice basis and let $g$ vary among metrics that satisfy $\sigma g = g$. Then condition~\ref{thm:sigmacompatible:cond:dueimmagine} in Theorem~\ref{thm:sigmacompatible} reduces  to $g$ satisfying
\begin{gather*}
\frac{g_{56}}{g_{13} g_{24}}=-5,\ \frac{g_{56}}{g_{13} g_{24}}=-5,\ \frac{g_{77}}{g_{13} g_{56}}=-3,\ \frac{g_{77}}{g_{13} g_{56}}=-3,\\
\frac{g_{77}}{g_{24}^2}+7=0,\ \frac{g_{88}}{g_{24} g_{56}}=-3,\ \frac{g_{88}}{g_{24} g_{56}}=-3,\ \frac{g_{88}}{g_{13}^2}+7=0.
\end{gather*}
These equations have the unique solution
\[g_{13}=-\frac{7}{15},\ g_{24}=-\frac{7}{15},\ g_{56}=-\frac{49}{45},\ g_{77}=-\frac{343}{225},\ g_{88}=-\frac{343}{225}.\]
\end{example}

It is sometimes more convenient to proceed in a different way, by assuming that the basis $\{e_i\}$ is orthonormal, so that the metric is defined by a vector $g\in(\Z^*)^n$, and allowing the structure constants to vary. Solving equation~\eqref{eqn:polynomial_eqns} amounts to solving
\[ M_{\Delta,2}\logsign g =\logsign x_I,\]
and imposing the Jacobi identity on the structure constants $c_I=\pm\sqrt{\abs{x_I}}$. Some of the sign ambiguities in the structure constants can be eliminated by choosing representatives up to equivalence.

\begin{example}\label{ex:842:117VariConst}
In the case of the Lie algebra \texttt{842:117} of Example~\ref{esempio:MaxRank}, the above procedure shows that the metric must have signature $(++++--++)$ and the structure constants are given by
\[(0,0,0,0,\pm\sqrt{5}e^{12},\pm\sqrt{5}e^{34},\pm\sqrt{3}e^{15}\pm\sqrt{7}e^{24}\pm\sqrt{3}e^{36},\pm\sqrt{7}e^{13}\pm\sqrt{3}e^{25}\pm\sqrt{3}e^{46}).\]
Since $M_{\Delta,2}$ is surjective, all sign ambiguities can be eliminated, i.e. each solution is equivalent to
\[(0,0,0,0,\sqrt{5}e^{12},\sqrt{5}e^{34},\sqrt{3}e^{15}+\sqrt{7}e^{24}+\sqrt{3}e^{36},\sqrt{7}e^{13}+\sqrt{3}e^{25}+\sqrt{3}e^{46}).\]
\end{example}

\begin{table}[thp]
{\setlength{\tabcolsep}{2pt}
\centering
\caption{8-dimensional nice Lie algebras with a diagonal Einstein metric\label{table:8DimNiceEinsteinDiagonal}}
\begin{footnotesize}
\begin{tabular}{>{\ttfamily}c C C C}
\toprule
\textnormal{Name} & \g & \text{Metric}&\text{Sign.}\\
\midrule
86532:6 & \multicolumn{1}{L}{0,0,\sqrt{5}e^{12},-\sqrt{\frac{5}{2}}e^{13},\sqrt{\frac{5}{2}}e^{23},-\sqrt{5}e^{15}+\sqrt{5}e^{24},} & (++---+--) & (3,5) \\ 
 & \multicolumn{1}{R}{\sqrt{\frac{11}{2}}e^{16}+2 \sqrt{3}e^{25}+\sqrt{\frac{11}{2}}e^{34},2 \sqrt{3}e^{14}+\sqrt{\frac{11}{2}}e^{26}+\sqrt{\frac{11}{2}}e^{35}} & & \\[5pt]
 & \multicolumn{1}{L}{0,0,4 \sqrt{3}e^{12},-\sqrt{\frac{5}{2}}e^{13},\sqrt{\frac{5}{2}}e^{23},-3 \sqrt{\frac{7}{2}}e^{15}+3 \sqrt{\frac{7}{2}}e^{24},}& (+++++-++) & (7,1)\\
 & \multicolumn{1}{R}{4 \sqrt{2}e^{16}+2 \sqrt{3}e^{25}+\sqrt{21}e^{34},2 \sqrt{3}e^{14}+4 \sqrt{2}e^{26}+\sqrt{21}e^{35}} & & \\[5pt]
8531:60a & \multicolumn{1}{L}{0,0,0,\frac{3}{\sqrt{2}}e^{12},2 \sqrt{2}e^{13},\sqrt{\frac{13}{2}}e^{24},} & (++++--++) & (6,2)\\ 
& \multicolumn{1}{R}{\sqrt{\frac{15}{2}}e^{15}+\sqrt{\frac{17}{2}}e^{23},\sqrt{10}e^{14}+\sqrt{\frac{15}{2}}e^{26}+\sqrt{\frac{3}{2}}e^{35}} & (++-++--+) & (5,3)\\[5pt]
8531:60b & \multicolumn{1}{L}{0,0,0,-\frac{3}{\sqrt{2}}e^{12},2 \sqrt{2}e^{13},\sqrt{\frac{13}{2}}e^{24},} & (++++--++) & (6,2)\\ 
 & \multicolumn{1}{R}{\sqrt{\frac{15}{2}}e^{15}+\sqrt{\frac{17}{2}}e^{23},\sqrt{10}e^{14}+\sqrt{\frac{15}{2}}e^{26}+\sqrt{\frac{3}{2}}e^{35}} & (++-++--+) & (5,3)\\[5pt]
842:117  & \multicolumn{1}{L}{0,0,0,0,\sqrt{5}e^{12},\sqrt{5}e^{34},}& (++++--++) & (6,2)\\ 
 & \multicolumn{1}{R}{\sqrt{3}e^{15}+\sqrt{7}e^{24}+\sqrt{3}e^{36},\sqrt{7}e^{13}+\sqrt{3}e^{25}+\sqrt{3}e^{46}} & &\\[5pt]
842:121a & \multicolumn{1}{L}{0,0,0,0,\sqrt{\frac{5}{2}}e^{13}-\sqrt{\frac{5}{2}}e^{24},\sqrt{\frac{5}{2}}e^{12}+\sqrt{\frac{5}{2}}e^{34},} & (++++--++) & (6,2)\\
 & \multicolumn{1}{R}{\sqrt{7}e^{14}+\sqrt{3}e^{25}+\sqrt{3}e^{36},\sqrt{3}e^{15}+\sqrt{7}e^{23}+\sqrt{3}e^{46}} & (++--+---) & (3,5)\\[5pt]
842:121b & \multicolumn{1}{L}{0,0,0,0,-\sqrt{\frac{5}{2}}e^{13}+\sqrt{\frac{5}{2}}e^{24},-\sqrt{\frac{5}{2}}e^{12}+\sqrt{\frac{5}{2}}e^{34},} & (++++--++) & (6,2)\\
 & \multicolumn{1}{R}{\sqrt{7}e^{14}+\sqrt{3}e^{25}+\sqrt{3}e^{36},\sqrt{3}e^{15}+\sqrt{7}e^{23}+\sqrt{3}e^{46}} & (++--+---) & (3,5)\\
\bottomrule
\end{tabular}
\end{footnotesize}
}
\end{table}

\begin{theorem}
\label{thm:EinsteinMetrics8}
A nice nilpotent Lie algebra of dimension $8$ has a diagonal Einstein metric of positive scalar curvature if and only if it appears in Table~\ref{table:8DimNiceEinsteinDiagonal}; it has a diagram involution $\sigma$ and a $\sigma$-diagonal  Einstein metric of positive scalar curvature if and only if it appears in Table~\ref{table:8DimNiceEinsteinSigma}. Both tables contain all metrics with this property up to rescaling.
\end{theorem}
\begin{proof}
Using the classification of nice nilpotent Lie algebras of dimension $8$ in \cite{ContiRossi:Construction} and applying Lemma~\ref{lemma:TracelessDerivations}, we find that those with traceless derivations are exactly those listed in Table~\ref{table:8DimNiceEinsteinDiagonal}. For each of these Lie algebras, we apply the method of Example~\ref{ex:842:117VariConst} to determine diagonal solutions of \eqref{eqn:normalized_einstein}.

To classify $\sigma$-diagonal solutions, for each Lie algebra we consider the elements $\sigma$ of order two in the group of automorphisms $\Aut(\Delta)$ and proceed as in Example~\ref{ex:8dimSIGMA}. For each entry in the table, we list the nonzero components of the invariant Einstein metric $(g_{ij})$.
\end{proof}
\begin{remark}
A case-by-case inspection shows that the nice Lie algebras appearing in Theorem~\ref{thm:EinsteinMetrics8} are equivalent to nice Lie algebras with rational structure constants; therefore, they determine Einstein metrics on a compact nilmanifold.
\end{remark}

\begin{table}[thp]
{\setlength{\tabcolsep}{2pt}
\centering
\begin{footnotesize}
\caption{8-dimensional nice Lie algebra with a $\sigma$-diagonal Einstein metric\label{table:8DimNiceEinsteinSigma}}
\begin{tabular}{>{\ttfamily}c C C}
\toprule
\textnormal{Name} & \multicolumn{2}{C}{\g} \\
$\sigma$ & \text{Metric} &\text{Sign.} \\
\midrule
86532:6 & \multicolumn{2}{C}{0,0,e^{12},- e^{13},e^{23},-e^{15}+e^{24}, e^{16}+e^{25}+e^{34},e^{14}+e^{26}+e^{35}} \\
\textnormal{(1,2)(4,5)(7,8)} & (g_{12},g_{33}, g_{45},g_{66},g_{78}) =\left(-\frac{24}{55},\frac{576}{605},\frac{6912}{6655},-\frac{165888}{73205},-\frac{1990656}{366025}\right)& (4,4)\\[2pt]
&(g_{12},g_{33}, g_{45},g_{66},g_{78}) =\left(-\frac{1}{84}, -\frac{1}{147},-\frac{5}{24696},\frac{5}{65856},\frac{5}{172872}\right)& (4,4)\\[8pt]
842:117  & \multicolumn{2}{C}{0,0,0,0,e^{12},e^{34},e^{15}+e^{24}+e^{36},e^{13}+e^{25}+e^{46}}\\
\textnormal{(1,3)(2,4)(5,6)}&(g_{13},g_{24},g_{56},g_{77},g_{88})=\left(-\frac{7}{15},-\frac{7}{15},-\frac{49}{45},-\frac{343}{225},-\frac{343}{225}\right)&(3,5)\\[5pt]
\textnormal{(1,2)(3,4)(7,8)}&(g_{12},g_{34},g_{55},g_{66},g_{78})=\left(-\frac{7}{15},-\frac{7}{15},\frac{49}{45},\frac{49}{45},\frac{343}{225}\right)&(5,3)\\[5pt]
\textnormal{(1,4)(2,3)(5,6)(7,8)}&(g_{14},g_{23},g_{56},g_{78})=\left(\frac{7}{15},\frac{7}{15},\frac{49}{45},-\frac{343}{225}\right)&(4,4)\\[8pt]
842:121a & \multicolumn{2}{C}{0,0,0,0,e^{13}-e^{24},e^{12}+e^{34},e^{14}+e^{25}+e^{36},e^{15}+e^{23}+e^{46}} \\
\textnormal{(1,2)(3,4)(7,8)}&(g_{12},g_{34},g_{55},g_{66},g_{78})=\left(-\frac{14}{15},-\frac{14}{15},\frac{98}{45},\frac{98}{45},\frac{1372}{225}\right) & (5,3)\\[2pt]
& (g_{12},g_{34},g_{55},g_{66},g_{78})=\left(-\frac{14}{15},\frac{14}{15},-\frac{98}{45},\frac{98}{45},-\frac{1372}{225}\right)&(4,4) \\[8pt]
842:121b & \multicolumn{2}{C}{0,0,0,0,-e^{13}+e^{24},-e^{12}+e^{34},e^{14}+e^{25}+e^{36},e^{15}+e^{23}+e^{46}} \\  
\textnormal{(1,2)(3,4)(7,8)}&(g_{12},g_{34},g_{55},g_{66},g_{78})=\left(-\frac{14}{15},-\frac{14}{15},\frac{98}{45},\frac{98}{45},\frac{1372}{225}\right)& (5,3)\\[2pt]
&(g_{12},g_{34},g_{55},g_{66},g_{78})=\left(-\frac{14}{15},\frac{14}{15},-\frac{98}{45},\frac{98}{45},-\frac{1372}{225}\right)&(4,4) \\
\bottomrule
\end{tabular}
\end{footnotesize}
}
\end{table}

In dimension $9$, there are $130$ nice diagrams that contain a Lie algebra that satisfies the necessary conditions~\ref{eqn:linear_system} and \ref{necessarycondition2} of corollary~\ref{cor:necessary_condition}. In addition, the computations involved in determining the actual Einstein metric become more complicated, as evident in the following:
\begin{example}
\label{esempio:MoreBrackets}
Take the Lie algebra \texttt{952:782}:
\[(0,0,0,0,e^{13}+e^{24},- e^{12},e^{34},e^{15}+e^{23}+e^{46},e^{14}+e^{27}+e^{35}).\]
It is easy to see that the matrix
\[M_\Delta=\left(
\begin{array}{ccccccccc}
 -1 & 0 & -1 & 0 & 1 & 0 & 0 & 0 & 0 \\
 0 & -1 & 0 & -1 & 1 & 0 & 0 & 0 & 0 \\
 -1 & -1 & 0 & 0 & 0 & 1 & 0 & 0 & 0 \\
 0 & 0 & -1 & -1 & 0 & 0 & 1 & 0 & 0 \\
 -1 & 0 & 0 & 0 & -1 & 0 & 0 & 1 & 0 \\
 0 & -1 & -1 & 0 & 0 & 0 & 0 & 1 & 0 \\
 0 & 0 & 0 & -1 & 0 & -1 & 0 & 1 & 0 \\
 -1 & 0 & 0 & -1 & 0 & 0 & 0 & 0 & 1 \\
 0 & -1 & 0 & 0 & 0 & 0 & -1 & 0 & 1 \\
 0 & 0 & -1 & 0 & -1 & 0 & 0 & 0 & 1 \\
\end{array}
\right)	;\] has maximal rank  over both  $\Z_2$ and $\R$. It turns out that $(1,\dotsc, 1)$ is in the span of the rows, and we can find two Einstein metrics:
\begin{gather*}
g_2=\frac{3}{16} \left(\pm\sqrt{249} g_1^2-9 g_1^2\right),\ g_3 = \frac{731 \pm 47 \sqrt{249}}{2205 g_1},\\
g_4=\frac{131253\pm8321 \sqrt{249}}{463050 g_1^2},\ g_5=\frac{-1}{735} \left(\pm47 \sqrt{249}+731\right),\\
g_6=\frac{-9}{16} \left(\pm5 \sqrt{249} g_1^3-73 g_1^3\right),\ g_7=\frac{-16 \left(\pm333103 \sqrt{249}+5256379\right)}{170170875 g_1^3},\\
g_8=\frac{2}{105} \left(\pm11 \sqrt{249} g_1-183 g_1\right),\ g_9=\frac{4 \left(8321 \sqrt{249}\pm131253\right)}{231525 g_1}.
\end{gather*}
Notice that the role of the parameter $g_1$ is apparent, as it originates from the one-dimensional group of diagonal Lie algebra automorphisms $\ker e^{M_\Delta}$, namely
\[\diag\left(k,k^2,\frac{1}{k},\frac{1}{k^2},1,k^3,\frac{1}{k^3},k,\frac{1}{k}\right).\]
\end{example}

\begin{remark}
The classification of nice nilpotent Lie algebras contains continuous families, starting from dimension $7$. In dimension up to $8$, each Lie algebra in a continuous family satisfies the obstruction of Lemma~\ref{lemma:TracelessDerivations}, and therefore admits no metric satisfying \eqref{eqn:normalized_einstein}. In dimension $9$, there are both completely obstructed and completely unobstructed families.

Consider for example the family of Lie algebra \texttt{952:399} 
\[(0,0,0,0,e^{12},e^{34}, a_1 e^{13}+e^{24},e^{14}+e^{25}+e^{36},e^{15}+e^{23}+e^{46})\]
where $a_1\neq0$. The elements of this family share the same root matrix:
\[M_\Delta=\left(
\begin{array}{ccccccccc}
 -1 & -1 & 0 & 0 & 1 & 0 & 0 & 0 & 0 \\
 0 & 0 & -1 & -1 & 0 & 1 & 0 & 0 & 0 \\
 -1 & 0 & -1 & 0 & 0 & 0 & 1 & 0 & 0 \\
 0 & -1 & 0 & -1 & 0 & 0 & 1 & 0 & 0 \\
 -1 & 0 & 0 & -1 & 0 & 0 & 0 & 1 & 0 \\
 0 & -1 & 0 & 0 & -1 & 0 & 0 & 1 & 0 \\
 0 & 0 & -1 & 0 & 0 & -1 & 0 & 1 & 0 \\
 -1 & 0 & 0 & 0 & -1 & 0 & 0 & 0 & 1 \\
 0 & -1 & -1 & 0 & 0 & 0 & 0 & 0 & 1 \\
 0 & 0 & 0 & -1 & 0 & -1 & 0 & 0 & 1 \\
\end{array}
\right).\] It is easy to see that the vector 
\[\tran{X}=(-6, -6, 1 - x_{247}, x_{247}, 8, -3 - x_{247}, -4 + x_{247}, -4 + x_{247}, 8, -3 - x_{247})\]
satisfies $\tran{X}M_{\Delta}=(1,\dots,1)$, and the corresponding system of condition~\ref{thm:sigmacompatible:cond:dueimmagine} of Theorem~\ref{thm:sigmacompatible} is given by:
\begin{gather*}
\frac{g_5}{g_1 g_2}=-6,\ \frac{g_6}{g_3 g_4}=-6,\ \frac{a_1^2 g_7}{g_1 g_3}+\frac{g_7}{g_2 g_4}=1,\\
\frac{g_8}{g_3 g_6}+4=\frac{g_7}{g_2 g_4},\ \frac{g_7}{g_4 g_2}+\frac{g_8}{g_5 g_2}+3=0,\ \frac{g_8}{g_1 g_4}=8,\\
\frac{g_7}{g_2 g_4}+\frac{g_9}{g_4 g_6}+3=0,\ \frac{g_9}{g_1 g_5}+4=\frac{g_7}{g_2 g_4},\ \frac{g_9}{g_2 g_3}=8.
\end{gather*}
After an easy computation, one obtains:
\begin{gather*}
g_2=\frac{4 g_1}{21 g_1-4},\ g_3=g_1,\ g_4=\frac{4 g_1}{21 g_1-4},\\
g_5=\frac{24 g_1^2}{4-21 g_1},\ g_6=\frac{24 g_1^2}{4-21 g_1},\ g_7=\frac{64 g_1 (3 g_1-1)}{3 (4-21 g_1)^2},\\
g_8=\frac{32 g_1^2}{21 g_1-4},\ g_9=\frac{32 g_1^2}{21 g_1-4},
\end{gather*}
where $a_1$ and $g_1$ must satisfy:
\[\frac{16 \left(a_1^2 (16-48 g_1)+3 g_1 (4-21 g_1)^2\right)}{(21 g_1-4)^3}=4.\]
We note that for any nonzero value of the parameter $a_1$  there exists a real solution with $\frac{1}{3}<g_1\leq\frac{4}{9}$. Summing up, we obtain a continuous family of Einstein Lie algebras with the same diagram $\Delta$.
\end{remark}

A more tractable class for a classification is the class of nice Lie algebras with surjective root matrix $M_\Delta$. With the methods of Theorem~\ref{thm:EinsteinMetrics8}, we obtain the following:
\begin{theorem}
\label{thm:EinsteinMetrics9}
A nice nilpotent Lie algebra of dimension $9$ with surjective root matrix has a diagonal Einstein metric of positive scalar curvature if and only if it appears in Table~\ref{table:9DimSurjectiveNormal}; it has a diagram involution $\sigma$ and a $\sigma$-diagonal Einstein metric of positive scalar curvature if and only if it appears in Table~\ref{table:9DimSIGMA}. Both tables contain all metrics with this property up to rescaling.
\end{theorem}
\begin{remark}
By \cite[Proposition 2.2]{ContiRossi:Construction}, any nice Lie algebra with surjective $M_\Delta$ is equivalent to a nice Lie algebra with structure constants equal to $\pm1$ or $0$. In particular, all the Lie groups constructed in Theorem~\ref{thm:EinsteinMetrics9} admit a compact Einstein quotient.
\end{remark}

\FloatBarrier
\begin{footnotesize}
{\setlength{\tabcolsep}{2pt}
\begin{longtable}[c]{>{\ttfamily}c C C C}
\caption{Einstein diagonal metrics on $9$-dimensional nice nilpotent Lie algebras with surjective $M_{\Delta}$\label{table:9DimSurjectiveNormal}}\\
\toprule
\textnormal{Name} & \g &\text{Metric} &\text{Sign.} \\
\midrule
\endfirsthead
\multicolumn{4}{c}{\tablename\ \thetable\ -- \textit{Continued from previous page}} \\
\toprule
\textnormal{Name} & \g &\text{Metric} &\text{Sign.} \\
\midrule
\endhead
\bottomrule\\[-7pt]
\multicolumn{4}{c}{\tablename\ \thetable\ -- \textit{Continued to next page}} \\
\endfoot
\bottomrule\\[-7pt]
\endlastfoot
9641:108a &\multicolumn{1}{L}{ 0,0,0,} &(++++-+-++) & (7,2)\\*
&\sqrt{\frac{11}{2}}e^{12},\sqrt{21}e^{13},3 \sqrt{2}e^{14},& (++-+++--+)& (6,3)\\*
&3 \sqrt{\frac{3}{2}}e^{24},\sqrt{\frac{41}{2}}e^{15}+\sqrt{\frac{43}{2}}e^{23},& (-++-++++-)&(6,3)\\*
&\multicolumn{1}{R}{\sqrt{17}e^{16}+\sqrt{\frac{29}{2}}e^{27}+\sqrt{\frac{3}{2}}e^{35}} & (-+---++--)&(3,6)\\[8pt]
9641:108b & \multicolumn{1}{L}{0,0,0,}& (++++-+-++)&(7,2)\\*
&\sqrt{\frac{11}{2}}e^{12},\sqrt{21}e^{13},-3 \sqrt{2}e^{14},& (++-+++--+)&(6,3)\\*
&3 \sqrt{\frac{3}{2}}e^{24},\sqrt{\frac{41}{2}}e^{15}+\sqrt{\frac{43}{2}}e^{23},& (-++-++++-) & (6,3)\\*
&\multicolumn{1}{R}{\sqrt{17}e^{16}+\sqrt{\frac{29}{2}}e^{27}+\sqrt{\frac{3}{2}}e^{35}} & (-+---++--) & (3,6)\\[8pt]
9641:108c & \multicolumn{1}{L}{0,0,0,} & (++++-+-++)&(7,2)\\*
&-\sqrt{\frac{11}{2}}e^{12},\sqrt{21}e^{13},3 \sqrt{2}e^{14},& (++-+++--+)&(6,3)\\*
&3 \sqrt{\frac{3}{2}}e^{24},\sqrt{\frac{41}{2}}e^{15}+\sqrt{\frac{43}{2}}e^{23},& (-++-++++-)&(6,3)\\*
&\multicolumn{1}{R}{\sqrt{17}e^{16}+\sqrt{\frac{29}{2}}e^{27}+\sqrt{\frac{3}{2}}e^{35}} & (-+---++--)&(3,6)\\[8pt]
9641:108d & \multicolumn{1}{L}{0,0,0,} & (++++-+-++)&(7,2)\\*
&-\sqrt{\frac{11}{2}}e^{12},\sqrt{21}e^{13},-3 \sqrt{2}e^{14},& (++-+++--+)& (6,3)\\*
&3 \sqrt{\frac{3}{2}}e^{24},\sqrt{\frac{41}{2}}e^{15}+\sqrt{\frac{43}{2}}e^{23},& (-++-++++-)&(6,3)\\*
&\multicolumn{1}{R}{\sqrt{17}e^{16}+\sqrt{\frac{29}{2}}e^{27}+\sqrt{\frac{3}{2}}e^{35}} & (-+---++--)&(3,6)\\[8pt]
9641:161 & \multicolumn{1}{L}{0,0,0,2 \sqrt{2}e^{12},\sqrt{\frac{14}{3}}e^{13},\sqrt{\frac{2}{3}}e^{24},\sqrt{7}e^{35},} & (+-++++--+)&(6,3)\\*
&\multicolumn{1}{R}{\frac{5}{\sqrt{3}}e^{14}+2 \sqrt{\frac{7}{3}}e^{23},4 \sqrt{\frac{2}{3}}e^{15}+\sqrt{\frac{5}{3}}e^{26}+2 \sqrt{2}e^{37}} & \\[8pt]
9631:66 & \multicolumn{1}{L}{0,0,0,4e^{12},\sqrt{\frac{77}{3}}e^{13},\sqrt{\frac{11}{3}}e^{23},\sqrt{17}e^{14},} & (+--+----+)&(3,6)\\*
&\multicolumn{1}{R}{\sqrt{\frac{73}{3}}e^{15}+\sqrt{\frac{70}{3}}e^{36},3 \sqrt{2}e^{17}+2 \sqrt{\frac{14}{3}}e^{26}+\frac{1}{\sqrt{3}}e^{35}} & \\[8pt]
9631:67 & \multicolumn{1}{L}{0,0,0,2 \sqrt{\frac{2}{3}}e^{12},\sqrt{\frac{77}{3}}e^{13},\sqrt{\frac{67}{3}}e^{23},} & (+++++--++)&(7,2)\\*
&\multicolumn{1}{C}{\sqrt{17}e^{14},\sqrt{\frac{17}{3}}e^{15}+2 \sqrt{\frac{14}{3}}e^{24}+\sqrt{\frac{70}{3}}e^{36},} & & \\*
&\multicolumn{1}{R}{3 \sqrt{2}e^{17}+\sqrt{19}e^{35}} & & \\[8pt]
9631:79 & \multicolumn{1}{L}{0,0,0,\frac{7}{\sqrt{3}}e^{12},\sqrt{\frac{10}{3}}e^{13},\sqrt{\frac{77}{3}}e^{23},\sqrt{17}e^{24},} & (+++--+++-)&(6,3)\\*
&\multicolumn{1}{R}{\frac{1}{\sqrt{3}}e^{14}+\sqrt{\frac{74}{3}}e^{26}+\sqrt{\frac{70}{3}}e^{35},\sqrt{19}e^{15}+3 \sqrt{2}e^{27}} & \\[8pt]
9631:85 & \multicolumn{1}{L}{0,0,0,\frac{5}{\sqrt{3}}e^{12},2 \sqrt{7}e^{13},\sqrt{\frac{77}{3}}e^{23},\sqrt{17}e^{24},} & (++++-+-++)&(7,2)\\*
&\multicolumn{1}{R}{\sqrt{\frac{73}{3}}e^{14}+\sqrt{\frac{70}{3}}e^{35},\sqrt{\frac{17}{3}}e^{15}+3 \sqrt{2}e^{27}+\sqrt{\frac{74}{3}}e^{36}} & \\[8pt]
9631:90 & \multicolumn{1}{L}{0,0,0,\sqrt{\frac{43}{3}}e^{12},2 \sqrt{7}e^{13},\sqrt{\frac{59}{3}}e^{23},} & (++++-+-++)&(7,2)\\*
&\multicolumn{1}{C}{\sqrt{17}e^{24},\sqrt{\frac{53}{3}}e^{15}+2 \sqrt{\frac{14}{3}}e^{36},} & \\*
&\multicolumn{1}{R}{\sqrt{\frac{91}{3}}e^{14}+3 \sqrt{2}e^{27}+\sqrt{\frac{34}{3}}e^{35}} & \\[8pt]
9631:96a & \multicolumn{1}{L}{0,0,0,\sqrt{5}e^{12},\sqrt{14}e^{13},\sqrt{15}e^{23},\sqrt{17}e^{24},} & (++++-+-++)&(7,2)\\*
&\multicolumn{1}{R}{\sqrt{13}e^{15}+\sqrt{14}e^{26},\sqrt{21}e^{14}+3 \sqrt{2}e^{27}+\sqrt{2}e^{35}} & (++-++---+)&(5,4)\\[8pt]
9631:96b & \multicolumn{1}{L}{0,0,0,\sqrt{5}e^{12},-\sqrt{14}e^{13},\sqrt{15}e^{23},\sqrt{17}e^{24},} & (++++-+-++)&(7,2)\\*
&\multicolumn{1}{R}{\sqrt{13}e^{15}+\sqrt{14}e^{26},\sqrt{21}e^{14}+3 \sqrt{2}e^{27}+\sqrt{2}e^{35}} & (++-++---+)&(5,4)\\[8pt]
9531:331 & \multicolumn{1}{L}{0,0,0,0,\sqrt{\frac{59}{3}}e^{12},\sqrt{\frac{77}{3}}e^{13},} & (+++--+++-)&(6,3)\\*
&\multicolumn{1}{C}{\sqrt{17}e^{15},2 \sqrt{7}e^{16}+\sqrt{\frac{11}{3}}e^{25}+\sqrt{\frac{70}{3}}e^{34},} & &\\*
&\multicolumn{1}{R}{3 \sqrt{2}e^{17}+\sqrt{\frac{67}{3}}e^{24}+\sqrt{\frac{10}{3}}e^{36}} & &\\[8pt]
9531:386 & \multicolumn{1}{L}{0,0,0,0,\frac{8}{\sqrt{15}}e^{12},\sqrt{\frac{77}{15}}e^{34},} & (++-+-+++-)&(6,3)\\*
& \multicolumn{1}{C}{2 \sqrt{\frac{2}{5}}e^{15},4 \sqrt{\frac{7}{15}}e^{14}+\sqrt{\frac{11}{3}}e^{25}+\sqrt{\frac{14}{5}}e^{36},} & &\\*
&\multicolumn{1}{R}{\sqrt{\frac{13}{5}}e^{17}+2 \sqrt{\frac{26}{15}}e^{23}+\sqrt{\frac{10}{3}}e^{46}} &\\[8pt]
953:168a & \multicolumn{1}{L}{0,0,0,0,\sqrt{\frac{35}{2}}e^{12},\sqrt{\frac{11}{2}}e^{13},\sqrt{17}e^{15}+3 \sqrt{2}e^{23},} & (+++--++-+)&(6,3)\\*
&\multicolumn{1}{R}{\frac{7}{2}e^{14}+\frac{3 \sqrt{5}}{2}e^{36},\frac{3 \sqrt{7}}{2}e^{16}+\sqrt{\frac{3}{2}}e^{25}+\frac{\sqrt{53}}{2}e^{34}} & (+-+-++--+)&(5,4)\\[8pt]
953:168b & \multicolumn{1}{L}{0,0,0,0,\sqrt{\frac{35}{2}}e^{12},\sqrt{\frac{11}{2}}e^{13},\sqrt{17}e^{15}+3 \sqrt{2}e^{23},} & (+++--++-+)&(6,3)\\*
&\multicolumn{1}{R}{-\frac{7}{2}e^{14}+\frac{3 \sqrt{5}}{2}e^{36},\frac{3 \sqrt{7}}{2}e^{16}+\sqrt{\frac{3}{2}}e^{25}+\frac{\sqrt{53}}{2}e^{34}} & (+-+-++--+)&(5,4)\\[8pt]
953:172 & \multicolumn{1}{L}{0,0,0,0,\sqrt{\frac{77}{3}}e^{12},\sqrt{\frac{41}{3}}e^{13},\sqrt{17}e^{15}+3 \sqrt{2}e^{23},} & (+++--++-+)&(6,3)\\*
&\multicolumn{1}{R}{\frac{7}{\sqrt{3}}e^{24}+\sqrt{\frac{46}{3}}e^{36},2 \sqrt{7}e^{16}+\sqrt{\frac{29}{3}}e^{25}+2 \sqrt{\frac{13}{3}}e^{34}} & \\[8pt]
953:234a & \multicolumn{1}{L}{0,0,0,0,\sqrt{6}e^{12},\sqrt{6}e^{34},\sqrt{\frac{11}{2}}e^{15}+\sqrt{\frac{13}{2}}e^{23},} & (++++--+++)&(7,2)\\*
&\multicolumn{1}{R}{\frac{\sqrt{13}}{2}e^{13}+\frac{3}{2}e^{46},\frac{\sqrt{29}}{2}e^{14}+\sqrt{\frac{3}{2}}e^{25}+\frac{\sqrt{19}}{2}e^{36}} & (+-+++--++)&(6,3)\\[8pt]
953:234b & \multicolumn{1}{L}{0,0,0,0,\sqrt{6}e^{12},\sqrt{6}e^{34},\sqrt{\frac{11}{2}}e^{15}+\sqrt{\frac{13}{2}}e^{23},} & (++++--+++)&(7,2)\\*
&\multicolumn{1}{R}{-\frac{\sqrt{13}}{2}e^{13}+\frac{3}{2}e^{46},\frac{\sqrt{29}}{2}e^{14}+\sqrt{\frac{3}{2}}e^{25}+\frac{\sqrt{19}}{2}e^{36}} & (+-+++--++)&(6,3)\\[8pt]
953:235a & \multicolumn{1}{L}{0,0,0,0,\sqrt{6}e^{12},\sqrt{6}e^{34},} & (++++--+++)&(7,2)\\*
&\multicolumn{1}{C}{\sqrt{\frac{11}{2}}e^{15}+\sqrt{\frac{13}{2}}e^{23},} & (+++--++-+)&(6,3)\\*
&\multicolumn{1}{C}{\sqrt{\frac{13}{2}}e^{14}+\sqrt{\frac{11}{2}}e^{36},} & (+-+++--++)&(6,3)\\*
&\multicolumn{1}{R}{2e^{13}+\sqrt{\frac{3}{2}}e^{25}+\sqrt{\frac{3}{2}}e^{46}} & (+-+-++--+)&(5,4)\\[8pt]
953:235b & \multicolumn{1}{L}{0,0,0,0,\sqrt{6}e^{12},\sqrt{6}e^{34},} & (++++--+++)&(7,2)\\*
&\multicolumn{1}{C}{\sqrt{\frac{11}{2}}e^{15}+\sqrt{\frac{13}{2}}e^{23},} & (+++--++-+)&(6,3)\\*
&\multicolumn{1}{C}{-\sqrt{\frac{13}{2}}e^{14}+\sqrt{\frac{11}{2}}e^{36},} & (+-+++--++)&(6,3)\\*
&\multicolumn{1}{R}{2e^{13}+\sqrt{\frac{3}{2}}e^{25}+\sqrt{\frac{3}{2}}e^{46}} & (+-+-++--+)&(5,4)\\[8pt]
953:235c & \multicolumn{1}{L}{0,0,0,0,\sqrt{6}e^{12},\sqrt{6}e^{34},} & (++++--+++)&(7,2)\\*
&\multicolumn{1}{C}{-\sqrt{\frac{11}{2}}e^{15}+\sqrt{\frac{13}{2}}e^{23},} & (+++--++-+)&(6,3)\\*
&\multicolumn{1}{C}{-\sqrt{\frac{13}{2}}e^{14}+\sqrt{\frac{11}{2}}e^{36},} & (+-+++--++)&(6,3)\\*
&\multicolumn{1}{R}{2e^{13}+\sqrt{\frac{3}{2}}e^{25}+\sqrt{\frac{3}{2}}e^{46}} & (+-+-++--+)&(5,4)\\[8pt]
9521:70a & \multicolumn{1}{L}{0,0,0,0,\sqrt{5}e^{12},} & (+++-+-+-+)&(6,3)\\*
&\multicolumn{1}{C}{\sqrt{\frac{29}{2}}e^{14}+3 \sqrt{\frac{3}{2}}e^{23},} & (++-+++--+)&(6,3)\\*
&\multicolumn{1}{C}{\sqrt{\frac{29}{2}}e^{13}+3 \sqrt{\frac{3}{2}}e^{24},} & (-+++-----)&(3,6)\\*
&\multicolumn{1}{R}{\sqrt{17}e^{15},3 \sqrt{2}e^{18}+\sqrt{21}e^{25}+\sqrt{2}e^{34}} & (-+---++--)&(3,6)\\[8pt]
9521:70b & \multicolumn{1}{L}{0,0,0,0,\sqrt{5}e^{12},} & (+++-+-+-+)&(6,3)\\*
&\multicolumn{1}{C}{-\sqrt{\frac{29}{2}}e^{14}+3 \sqrt{\frac{3}{2}}e^{23},} & (++-+++--+)&(6,3)\\*
&\multicolumn{1}{C}{\sqrt{\frac{29}{2}}e^{13}+3 \sqrt{\frac{3}{2}}e^{24},} & (-+++-----)&(3,6)\\*
&\multicolumn{1}{R}{\sqrt{17}e^{15},3 \sqrt{2}e^{18}+\sqrt{21}e^{25}+\sqrt{2}e^{34}} & (-+---++--)&(3,6)\\[8pt]
9521:70c & \multicolumn{1}{L}{0,0,0,0,-\sqrt{5}e^{12},} & (+++-+-+-+)&(6,3)\\*
&\multicolumn{1}{C}{-\sqrt{\frac{29}{2}}e^{14}+3 \sqrt{\frac{3}{2}}e^{23},} & (++-+++--+)&(6,3)\\*
&\multicolumn{1}{C}{\sqrt{\frac{29}{2}}e^{13}+3 \sqrt{\frac{3}{2}}e^{24},} & (-+++-----)&(3,6)\\*
&\multicolumn{1}{R}{\sqrt{17}e^{15},3 \sqrt{2}e^{18}+\sqrt{21}e^{25}+\sqrt{2}e^{34}} & (-+---++--)&(3,6)\\[8pt]
952:240 & \multicolumn{1}{L}{0,0,0,0,2 \sqrt{2}e^{12},2 \sqrt{5}e^{13},\sqrt{17}e^{14}+3 \sqrt{2}e^{23},} & (+++-+-+-+)&(6,3)\\*
&\multicolumn{1}{R}{\sqrt{21}e^{25}+\sqrt{22}e^{34},2 \sqrt{7}e^{15}+\sqrt{6}e^{24}+\sqrt{21}e^{36}} & \\[8pt]
952:246 & \multicolumn{1}{L}{0,0,0,0,\sqrt{14}e^{12},\sqrt{2}e^{13},\sqrt{17}e^{14}+3 \sqrt{2}e^{23},} & (+++--+++-)&(6,3)\\*
&\multicolumn{1}{R}{\sqrt{22}e^{16}+\sqrt{21}e^{25},\sqrt{6}e^{15}+4e^{24}+\sqrt{21}e^{36}} & \\[8pt]
952:355 & \multicolumn{1}{L}{0,0,0,0,\sqrt{6}e^{12},\sqrt{6}e^{34},2e^{13}+\sqrt{3}e^{24},} & (+++--+++-)&(6,3)\\*
&\multicolumn{1}{R}{\sqrt{7}e^{15}+2 \sqrt{2}e^{23},2 \sqrt{2}e^{14}+\sqrt{7}e^{36}} & (+-+++-+-+)&(6,3)\\[8pt]
952:394 & \multicolumn{1}{L}{0,0,0,0,\sqrt{6}e^{12},\sqrt{6}e^{34},e^{13},} & (++++--+++)&(7,2)\\*
&\multicolumn{1}{R}{2 \sqrt{2}e^{14}+\sqrt{3}e^{25}+2e^{36},2e^{15}+2 \sqrt{2}e^{23}+\sqrt{3}e^{46}} & &\\[8pt]
952:724 & \multicolumn{1}{L}{0,0,0,0,2 \sqrt{5}e^{12},\sqrt{17}e^{13},\sqrt{14}e^{14},} & (++-+++--+)&(6,3)\\*
&\multicolumn{1}{R}{3 \sqrt{2}e^{16}+3e^{25}+2 \sqrt{7}e^{34},2 \sqrt{7}e^{15}+2 \sqrt{3}e^{23}+\sqrt{15}e^{47}} & &\\[8pt]
952:725a & \multicolumn{1}{L}{0,0,0,0,} & (++++---++)&(6,3)\\*
&\multicolumn{1}{C}{\sqrt{3}e^{12},2e^{13},2e^{24},} & (+-+++-++-)&(6,3)\\*
&\multicolumn{1}{C}{2 \sqrt{2}e^{14}+\sqrt{2}e^{25}+\sqrt{5}e^{36},} & (-+++++--+)&(6,3)\\*
&\multicolumn{1}{R}{\sqrt{2}e^{15}+2 \sqrt{2}e^{23}+\sqrt{5}e^{47}} & (--++-++--)&(4,5)\\[8pt]
952:725b & \multicolumn{1}{L}{0,0,0,0,} & (++++---++)&(6,3)\\*
&\multicolumn{1}{C}{\sqrt{3}e^{12},-2e^{13},2e^{24},} & (+-+++-++-)&(6,3)\\*
&\multicolumn{1}{C}{2 \sqrt{2}e^{14}+\sqrt{2}e^{25}+\sqrt{5}e^{36},} & (-+++++--+)&(6,3)\\*
&\multicolumn{1}{R}{\sqrt{2}e^{15}+2 \sqrt{2}e^{23}+\sqrt{5}e^{47}} & (--++-++--)&(4,5)\\[8pt]
952:725c & \multicolumn{1}{L}{0,0,0,0,} & (++++---++)&(6,3)\\*
&\multicolumn{1}{C}{-\sqrt{3}e^{12},-2^{13},2e^{24},} & (+-+++-++-)&(6,3)\\*
&\multicolumn{1}{C}{2 \sqrt{2}e^{14}+\sqrt{2}e^{25}+\sqrt{5}e^{36},} & (-+++++--+)&(6,3)\\*
&\multicolumn{1}{R}{\sqrt{2}e^{15}+2 \sqrt{2}e^{23}+\sqrt{5}e^{47}} & (--++-++--)&(4,5)\\[8pt]
952:727a & \multicolumn{1}{L}{0,0,0,0,\sqrt{\frac{33}{2}}e^{12},\sqrt{17}e^{13},\sqrt{\frac{21}{2}}e^{24},} & (++-+++--+)&(6,3)\\*
&\multicolumn{1}{C}{3 \sqrt{2}e^{16}+\sqrt{2}e^{25}+\sqrt{21}e^{34},} & (+--+-++--)&(4,5)\\*
&\multicolumn{1}{R}{\sqrt{\frac{35}{2}}e^{15}+\sqrt{5}e^{23}+\sqrt{\frac{23}{2}}e^{47}} & & \\[8pt]
952:727b & \multicolumn{1}{L}{0,0,0,0,-\sqrt{\frac{33}{2}}e^{12},\sqrt{17}e^{13},\sqrt{\frac{21}{2}}e^{24},} & (++-+++--+)&(6,3)\\*
&\multicolumn{1}{C}{3 \sqrt{2}e^{16}+\sqrt{2}e^{25}+\sqrt{21}e^{34},} & (+--+-++--)&(4,5)\\*
&\multicolumn{1}{R}{\sqrt{\frac{35}{2}}e^{15}+\sqrt{5}e^{23}+\sqrt{\frac{23}{2}}e^{47}} & & \\[8pt]
952:730a & \multicolumn{1}{L}{0,0,0,0,\sqrt{5}e^{12},\sqrt{2}e^{13},2e^{34},} & (++++---++)&(6,3)\\*
&\multicolumn{1}{R}{2 \sqrt{2}e^{14}+2e^{25}+\sqrt{3}e^{36},\sqrt{2}e^{15}+2 \sqrt{2}e^{23}+\sqrt{5}e^{47}} & (-+++++--+)&(6,3) \\[8pt]
952:730b & \multicolumn{1}{L}{0,0,0,0,-\sqrt{5}e^{12},\sqrt{2}e^{13},2e^{34},} & (++++---++)&(6,3)\\*
&\multicolumn{1}{R}{2 \sqrt{2}e^{14}+2e^{25}+\sqrt{3}e^{36},\sqrt{2}e^{15}+2 \sqrt{2}e^{23}+\sqrt{5}e^{47}} & (-+++++--+)&(6,3) \\[8pt]
952:737 & \multicolumn{1}{L}{0,0,0,0,\sqrt{\frac{77}{3}}e^{12},\sqrt{\frac{59}{3}}e^{13},\sqrt{17}e^{24},} & (-+-+----+)&(3,6)\\*
&\multicolumn{1}{R}{\sqrt{\frac{70}{3}}e^{16}+\frac{5}{\sqrt{3}}e^{25}+4e^{34},\frac{7}{\sqrt{3}}e^{15}+3 \sqrt{2}e^{27}+2 \sqrt{\frac{2}{3}}e^{36}} & & \\[8pt]
952:738a & \multicolumn{1}{L}{0,0,0,0,\sqrt{11}e^{12},2 \sqrt{2}e^{13},2 \sqrt{2}e^{14},} & (+-++++--+)&(6,3)\\*
&\multicolumn{1}{R}{\sqrt{6}e^{15}+4e^{24}+3e^{36},4e^{16}+\sqrt{6}e^{25}+3e^{47}} & (-+-+++++-)&(6,3) \\[8pt]
952:738b & \multicolumn{1}{L}{0,0,0,0,-\sqrt{11}e^{12},2 \sqrt{2}e^{13},2 \sqrt{2}e^{14},} & (+-++++--+)&(6,3)\\*
&\multicolumn{1}{R}{\sqrt{6}e^{15}+4e^{24}+3e^{36},4e^{16}+\sqrt{6}e^{25}+3e^{47}} & (-+-+++++-)&(6,3) \\[8pt]
952:744a & \multicolumn{1}{L}{0,0,0,0,4 \sqrt{\frac{2}{3}}e^{12},\sqrt{\frac{41}{3}}e^{13},2 \sqrt{2}e^{24},} & (++-+++--+)&(6,3)\\*
&\multicolumn{1}{R}{\sqrt{\frac{34}{3}}e^{16}+\sqrt{\frac{11}{3}}e^{25}+4e^{34},2 \sqrt{\frac{10}{3}}e^{15}+\sqrt{\frac{10}{3}}e^{36}+3e^{47}} & (--++++++-)&(6,3) \\[8pt]
952:744b & \multicolumn{1}{L}{0,0,0,0,-4 \sqrt{\frac{2}{3}}e^{12},\sqrt{\frac{41}{3}}e^{13},2 \sqrt{2}e^{24},} & (++-+++--+)&(6,3)\\*
&\multicolumn{1}{R}{\sqrt{\frac{34}{3}}e^{16}+\sqrt{\frac{11}{3}}e^{25}+4e^{34},2 \sqrt{\frac{10}{3}}e^{15}+\sqrt{\frac{10}{3}}e^{36}+3e^{47}} & (--++++++-)&(6,3) \\[8pt]
952:750 & \multicolumn{1}{L}{0,0,0,0,\sqrt{22}e^{12},\sqrt{14}e^{34},5e^{13},} & (--++----+)&(3,6)\\*
&\multicolumn{1}{R}{2 \sqrt{7}e^{15}+4e^{23}+\sqrt{13}e^{46},2 \sqrt{6}e^{17}+\sqrt{5}e^{25}+2 \sqrt{7}e^{36}} & & \\[8pt]
952:753a & \multicolumn{1}{L}{0,0,0,0,\sqrt{\frac{41}{3}}e^{12},\sqrt{\frac{59}{6}}e^{34},} & (++++--+++)&(7,2)\\*
&\multicolumn{1}{C}{\frac{5}{\sqrt{2}}e^{13},\sqrt{\frac{34}{3}}e^{15}+4e^{24}+\sqrt{\frac{11}{3}}e^{36},} & (-+-+++++-)&(6,3)\\*
&\multicolumn{1}{R}{\sqrt{\frac{23}{2}}e^{17}+\sqrt{\frac{10}{3}}e^{25}+\sqrt{\frac{43}{6}}e^{46}} & &\\[8pt]
952:753b & \multicolumn{1}{L}{0,0,0,0,\sqrt{\frac{41}{3}}e^{12},\sqrt{\frac{59}{6}}e^{34},} & (++++--+++)&(7,2)\\*
&\multicolumn{1}{C}{-\frac{5}{\sqrt{2}}e^{13},\sqrt{\frac{34}{3}}e^{15}+4e^{24}+\sqrt{\frac{11}{3}}e^{36},} & (-+-+++++-)&(6,3)\\*
&\multicolumn{1}{R}{\sqrt{\frac{23}{2}}e^{17}+\sqrt{\frac{10}{3}}e^{25}+\sqrt{\frac{43}{6}}e^{46}} & &\\[8pt]
942:110a & \multicolumn{1}{L}{0,0,0,0,0,2 \sqrt{3}e^{12},\sqrt{\frac{19}{2}}e^{13}+\sqrt{\frac{17}{2}}e^{45},} & (++-++--++)&(6,3)\\*
&\multicolumn{1}{C}{\sqrt{\frac{29}{3}}e^{16}+\sqrt{\frac{43}{3}}e^{25}+\sqrt{\frac{11}{3}}e^{34},} & (--+++----)&(3,6)\\*
&\multicolumn{1}{R}{\sqrt{\frac{67}{6}}e^{14}+\sqrt{\frac{10}{3}}e^{26}+\sqrt{\frac{41}{6}}e^{35}} & &\\[8pt]
942:110b & \multicolumn{1}{L}{0,0,0,0,0,2 \sqrt{3}e^{12},-\sqrt{\frac{19}{2}}e^{13}+\sqrt{\frac{17}{2}}e^{45},} & (++-++--++)&(6,3)\\*
&\multicolumn{1}{C}{\sqrt{\frac{29}{3}}e^{16}+\sqrt{\frac{43}{3}}e^{25}+\sqrt{\frac{11}{3}}e^{34},} & (--+++----)&(3,6)\\*
&\multicolumn{1}{R}{\sqrt{\frac{67}{6}}e^{14}+\sqrt{\frac{10}{3}}e^{26}+\sqrt{\frac{41}{6}}e^{35}} & &\\[8pt]
942:112 & \multicolumn{1}{L}{0,0,0,0,0,2 \sqrt{3}e^{12},3 \sqrt{2}e^{13}+\sqrt{17}e^{24},} & (-+--+++++)&(6,3)\\*
&\multicolumn{1}{C}{\sqrt{21}e^{16}+2 \sqrt{5}e^{25}+\sqrt{2}e^{34},} & &\\*
&\multicolumn{1}{R}{\sqrt{14}e^{14}+2 \sqrt{2}e^{26}+\sqrt{21}e^{35}} & &\\*
\end{longtable}
}
\end{footnotesize}

\FloatBarrier
\begin{footnotesize}
{\setlength{\tabcolsep}{2pt}
\begin{longtable}[c]{>{\ttfamily}c C C}
\caption{
$\sigma$-diagonal Einstein metrics on $9$-dimensional nice nilpotent Lie algebras with surjective $M_{\Delta}$\label{table:9DimSIGMA}}\\
\toprule
\textnormal{Name} & \multicolumn{2}{C}{\g}\\
$\sigma$ & \text{Metric} &\text{Sign.} \\
\midrule
\endfirsthead
\multicolumn{3}{c}{\tablename\ \thetable\ -- \textit{Continued from previous page}} \\
\toprule
\textnormal{Name} & \multicolumn{2}{C}{\g}\\
$\sigma$ & \text{Metric} &\text{Sign.} \\
\midrule
\endhead
\bottomrule\\[-7pt]
\multicolumn{3}{c}{\tablename\ \thetable\ -- \textit{Continued to next page}} \\
\endfoot
\bottomrule\\[-7pt]
\endlastfoot
\texttt{953:235a} & \multicolumn{2}{C}{0,0,0,0,e^{12},e^{34},e^{15}+e^{23},e^{14}+e^{36},e^{13}+e^{25}+e^{46}} \\*
\textnormal{(1,3)(2,4)(5,6)(7,8)} & (g_{13},g_{24}, g_{56},g_{78},g_9) &\\*
&\left(-\frac{13}{66},-\frac{1}{3}\sqrt{\frac{26}{33}},-\frac{13}{33} \sqrt{\frac{26}{33}},-\frac{169}{198} \sqrt{\frac{13}{66}},-\frac{169}{1089}\right)& (4,5)\\*[2pt]
&\left(-\frac{13}{66},\frac{1}{3}\sqrt{\frac{26}{33}},\frac{13}{33}\sqrt{\frac{26}{33}},\frac{169}{198} \sqrt{\frac{13}{66}},-\frac{169}{1089}\right)& (4,5)\\[10pt]
\texttt{953:235c} & \multicolumn{2}{C}{0,0,0,0,e^{12},e^{34},-e^{15}+e^{23},-e^{14}+e^{36},e^{13}+e^{25}+e^{46}} \\*
\textnormal{(1,3)(2,4)(5,6)(7,8)} & (g_{13},g_{24}, g_{56},g_{78},g_9) &\\*
&\left(-\frac{13}{66},-\frac{1}{3}\sqrt{\frac{26}{33}},-\frac{13}{33} \sqrt{\frac{26}{33}},\frac{169}{198} \sqrt{\frac{13}{66}},-\frac{169}{1089}\right) & (4,5)\\*[2pt]
&\left(-\frac{13}{66},\frac{1}{3}\sqrt{\frac{26}{33}},\frac{13}{33} \sqrt{\frac{26}{33}},-\frac{169}{198} \sqrt{\frac{13}{66}},-\frac{169}{1089}\right)&(4,5)\\[10pt]
9521:70b & \multicolumn{2}{C}{0,0,0,0,e^{12},-e^{14}+e^{23},e^{24}+e^{13},e^{15},e^{18}+e^{25}+e^{34}} \\*
\textnormal{(3,4)(6,7)}& (g_{1},g_{2},g_{34}, g_{5},g_{67},g_{8},g_9) &\\*
\multicolumn{2}{C}{\left(\frac{203}{2754},\frac{5887}{74358},-\frac{41209 \sqrt{\frac{145}{51}}}{446148},\frac{5975305}{204781932},\frac{242597383 \sqrt{\frac{145}{51}}}{2457383184},-\frac{1212986915}{33174672984},\frac{246236343745}{5075724966552}\right)} & (6,3)\\*[2pt]
\multicolumn{2}{C}{ \left(\frac{203}{2754},\frac{5887}{74358},\frac{41209 \sqrt{\frac{145}{51}}}{446148},\frac{5975305}{204781932},-\frac{242597383 \sqrt{\frac{145}{51}}}{2457383184},-\frac{1212986915}{33174672984},\frac{246236343745}{5075724966552}\right)}&(6,3)\\[10pt]
9521:70c & \multicolumn{2}{C}{0,0,0,0,- e^{12},-e^{14}+e^{23},e^{13}+e^{24},e^{15},e^{18}+e^{25}+e^{34}} \\*
\textnormal{(3,4)(6,7)}& (g_{1},g_{2},g_{34}, g_{5},g_{67},g_{8},g_9) &\\*
\multicolumn{2}{C}{\left(\frac{203}{2754},\frac{5887}{74358},-\frac{41209 \sqrt{\frac{145}{51}}}{446148},\frac{5975305}{204781932},\frac{242597383 \sqrt{\frac{145}{51}}}{2457383184},-\frac{1212986915}{33174672984},\frac{246236343745}{5075724966552}\right)} & (6,3)\\*[2pt]
\multicolumn{2}{C}{ \left(\frac{203}{2754},\frac{5887}{74358},\frac{41209 \sqrt{\frac{145}{51}}}{446148},\frac{5975305}{204781932},-\frac{242597383 \sqrt{\frac{145}{51}}}{2457383184},-\frac{1212986915}{33174672984},\frac{246236343745}{5075724966552}\right)}&(6,3)\\[10pt]
952:394 & \multicolumn{2}{C}{0,0,0,0,e^{12},e^{34},e^{13},e^{14}+e^{25}+e^{36},e^{15}+e^{23}+e^{46}} \\* 
\textnormal{(1,3)(2,4)(5,6)(8,9)}& (g_{13},g_{24},g_{56},g_7,g_{89}) &\\*
&\left(-\frac{1}{3},-\frac{4}{9},-\frac{8}{9},-\frac{1}{9},-\frac{32}{27}\right) & (4,5)\\[10pt]
952:725b & \multicolumn{2}{C}{0,0,0,0,e^{12},- e^{13},e^{24},e^{14}+e^{25}+e^{36},e^{15}+e^{23}+e^{47}} \\* 
\textnormal{(1,2)(3,4)(6,7)(8,9)}& (g_{12},g_{34},g_{5},g_{67},g_{89}) &\\*
&\left(-2 \sqrt{\frac{2}{15}},-\frac{2}{5},\frac{8}{5},\frac{16}{5} \sqrt{\frac{2}{15}},\frac{32}{5} \sqrt{\frac{2}{15}}\right) & (5,4)\\*[2pt]
&\left(2 \sqrt{\frac{2}{15}},-\frac{2}{5},\frac{8}{5},-\frac{16}{5} \sqrt{\frac{2}{15}},-\frac{32}{5} \sqrt{\frac{2}{15}}\right) & (5,4)\\[10pt]
952:725c & \multicolumn{2}{C}{0,0,0,0,-e^{12},- e^{13},e^{24},e^{14}+e^{25}+e^{36},e^{15}+e^{23}+e^{47}} \\* 
\textnormal{(1,2)(3,4)(6,7)(8,9)}& (g_{12},g_{34},g_{5},g_{67},g_{89}) &\\*
&\left(-2 \sqrt{\frac{2}{15}},-\frac{2}{5},\frac{8}{5},\frac{16}{5} \sqrt{\frac{2}{15}},\frac{32}{5} \sqrt{\frac{2}{15}}\right) & (5,4)\\*[2pt]
&\left(2 \sqrt{\frac{2}{15}},-\frac{2}{5},\frac{8}{5},-\frac{16}{5} \sqrt{\frac{2}{15}},-\frac{32}{5} \sqrt{\frac{2}{15}}\right) & (5,4)\\
\end{longtable}
}
\end{footnotesize}

\bibliographystyle{plain}
\bibliography{nice}

\small\noindent Dipartimento di Matematica e Applicazioni, Universit\`a di Milano Bicocca, via Cozzi 55, 20125 Milano, Italy.\\
\texttt{diego.conti@unimib.it}\\
\texttt{federico.rossi@unimib.it}

\end{document}